\documentclass[a4paper,11pt,reqno]{amsart}
\usepackage{graphicx}
\usepackage{amsmath}
\usepackage{amssymb}
\usepackage{oldgerm}
\usepackage[all]{xy}
\usepackage{amsthm}
\usepackage{mathtools}
\newtheoremstyle{dotless}{}{}{\itshape}{}{\bfseries}{}{ }{}
\theoremstyle{dotless}
\newtheorem{te}{Theorem}[section]
\newtheorem{co}[te]{Corollary}
\newtheorem{re}[te]{Remark}

\newtheorem{de}[te]{\sc Definition}
\newtheorem{ex}[te]{Example}
\newtheorem{prop}[te]{Proposition}

\def\squareforqed{\hbox{\rlap{$\sqcap$}$\sqcup$}}
\def\qed{\ifmmode\else\unskip\quad\fi\squareforqed}
\def\smartqed{\def\qed{\ifmmode\squareforqed\else{\unskip\nobreak\hfil
\penalty50\hskip1em\null\nobreak\hfil\squareforqed
\parfillskip=0pt\finalhyphendemerits=0\endgraf}\fi}}
\smartqed

\usepackage{tikz}
\usetikzlibrary{calc,decorations.markings,arrows}
\tikzset{->-/.style={decoration={
  markings,
  mark=at position .5 with {\arrow{>}}},postaction={decorate}}}

\newlength{\qthickness}
\usepackage{color}
\usepackage{enumerate}
\usepackage{xy}
\usepackage{array}
\usepackage{amssymb}
\usepackage{psfrag}
\usepackage{graphics}
\usepackage{graphicx}
\usepackage{epsfig}

\usepackage{epic}
\usepackage{subfigure}
\usepackage{color}

\xyoption{matrix}
\xyoption{all}

\begin{document}
\baselineskip=15pt
\title[Extensions between Cohen-Macaulay modules]{Extensions between Cohen-Macaulay modules of Grassmannian cluster categories}

\author[Baur]{Karin Baur}
\email{baurk@uni-graz.at}

\author[Bogdanic]{Dusko Bogdanic}
\email{dusko.bogdanic@gmail.com}


\begin{abstract}  
In this paper we study extensions between  Cohen-Macaulay modules for algebras arising in the categorifications of Grassmannian cluster algebras. We prove that rank 1 modules are periodic, and we give explicit formulas for the computation of the period based solely on the rim of the rank 1 module in question. We determine ${\rm Ext}^i(L_I, L_J)$ for arbitrary rank 1 modules $L_I$ and $L_J$. An explicit combinatorial algorithm is given for computation of ${\rm Ext}^i(L_I, L_J)$ when $i$ is odd, and for $i$ even, we show that ${\rm Ext}^i(L_I, L_J)$ is cyclic over the centre, and we give an explicit formula for its computation. At the end of the paper we give a vanishing condition of ${\rm Ext}^i(L_I, L_J)$ for any $i>0$.

\end{abstract}

\maketitle \vspace{-10mm}

\section{Introduction and preliminaries}

In his study~\cite{POS} of the total positivity of the Grassmannian Gr$(k,n)$ of $k$-planes 
in $\mathbb{C}^n$, Postnikov 
introduced alternating strand diagrams as collections of $n$ curves in a disk satisfying 
certain axioms. 
Alternating strand diagrams associated to the permutation $i\mapsto i+k$ of $\{1,2,\dots, n\}$ were used by 
Scott~\cite{SCT} to show that the homogeneous coordinate ring of Gr$(k,n)$ has the structure 
of a cluster algebra, with each such diagram corresponding to a seed whose (extended) 
cluster consists of minors (i.e.\ of Pl\"ucker coordinates), where the minors are labelled by 
$k$-subsets of $\{1,2,\dots, n\}$. The diagram both gives the quiver of the cluster and the minors 
(cluster variables) contained in it: every alternating region of the diagram obtains as a label 
the $k$-subset formed by the strands passing to the right of the region, and the quiver can be read-off 
from the geometry of the strands. 
Oh-Postnikov-Speyer have proved in~\cite{OPS} that every cluster consisting of minors arises 
in this way, so there is a bijection between clusters of minors and strand diagrams for the Grassmann 
permutation. 
A categorification of this cluster algebra structure has been obtained by 
Geiss-Leclerc-Schroeer~\cite{GLS} via (a subcategory of) the category of finite dimensional 
modules over the preprojective algebra of type A$_{n-1}$.

In \cite{JKS}, Jensen-King-Su gave a new and extended categorification of this cluster structure using the maximal Cohen-Macaulay modules \cite{CMM} over the completion of an algebra $B$ which is a quotient of the preprojective algebra of type $\tilde{A}_{n-1}$. 
In particular, a rank $1$ Cohen-Macaulay $B$-module $L_I$ is associated to each $k$-subset $I$ of $\{1,2,\dots,n\}$.

It was shown in \cite{JKS} that every rigid indecomposable Cohen-Macaulay module for the above mentioned algebra $B$  has a generic filtration by rank 1 modules. This  enables a description of these modules in terms of the so-called  profiles, given by  collections of  $k$-subsets that correspond to the rank 1 modules in this filtration. In particular, this profile
determines the class of the module in the Grothendieck group of the category of Cohen-Macaulay modules. Therefore, rank 1 modules are the building blocks of the category of Cohen-Macaulay modules, and in order to understand representation-theoretic invariants for the category of all Cohen-Macaulay modules, we must first do so for rank 1 modules. Since the algebra $B$ is infinite dimensional, most of the homological computations are difficult to conduct, but for some problems it is possible to give complete answers. Such a problem is the computation of the extension spaces between rank 1  Cohen-Macaulay modules.

After some introductory remarks, in the second section of this paper we prove that rank 1 Cohen-Macaulay modules over the above mentioned completion of the algebra $B$ are periodic, with periods being even numbers in the case when $I$ is a disjoint union of more than two intervals. We give an explicit combinatorial formula for computation of the period of a given rank 1 module $L_I$ only in terms of the $k$-subset $I$, which is called the rim of the rank 1 module $L_I$. In the last section of this paper we give an explicit combinatorial description of the Ext-spaces between rank 1 Cohen-Macaulay modules. The description is  in terms of a new combinatorial and geometric construction consisting of a sequence of trapezia given by the rims of rank 1 Cohen-Macaulay modules. An explicit algorithm is constructed for the computation of the Ext-spaces which turn out to be finite dimensional. Also, we prove directly that the Ext-functor is commutative for rank 1 modules, and that  ${\rm Ext}^2(L_I,L_J)$, where $L_I$ and $L_J$ are rank 1 Cohen-Macaulay modules,  is a cyclic module over the centre $\mathbb F[t]$ of $B$. By using the fact that rank 1 modules are periodic,  it was proven that for any $i>0$, ${\rm Ext}^i(L_I,L_J)$ is a finite dimensional vector space. At the end of the paper we give a combinatorial criterion for vanishing of ${\rm Ext}^i(L_I, L_J)$ for any $i>0$.

\subsection{Notation and set-up} We now follow the exposition from \cite{BKM} in order to introduce notation and background results. 
Let $C$ be  a circular graph  with vertices $C_0=\mathbb Z_n$ set clockwise around a
circle, and with the set of edges, $C_1$, also labelled by $\mathbb Z_n$, with edge $i$ joining vertices $i-1$ and $i$. For integers $a,b\in \{1,2,\dots, n\}$, we denote by $[a,b]$ the closed cyclic
interval consisting of the elements of the set $\{a, a +1, \dots, b\}$ reduced modulo $ n$.
Consider the quiver with vertices $C_0$
and, for each edge $i\in C_1$, a pair of arrows $x_i{\colon}i-1\to i$
and $y_i{\colon}i\to i-1$. Then let $B$ be the quotient of the
path algebra (over $\mathbb{F}$, where $\mathbb F=\bar{\mathbb F}$) of this quiver by the ideal generated by the
$2n$ relations $xy=yx$ and $x^k=y^{n-k}$, interpreting $x$ and $y$
as arrows of the form $x_i,y_i$ appropriately and starting at any vertex, e.g.\  when $n=5$ we have

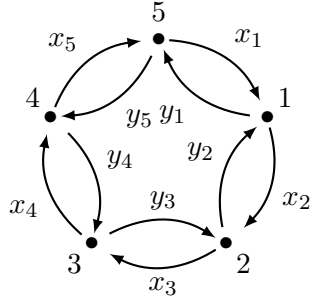
\begin{center}
\begin{tikzpicture}[scale=1]
\newcommand{\radius}{1.5cm}
\foreach \j in {1,...,5}{
  \path (90-72*\j:\radius) node[black] (w\j) {$\bullet$};
  \path (162-72*\j:\radius) node[black] (v\j) {};
  \path[->,>=latex] (v\j) edge[black,bend left=30,thick] node[black,auto] {$x_{\j}$} (w\j);
  \path[->,>=latex] (w\j) edge[black,bend left=30,thick] node[black,auto] {$y_{\j}$}(v\j);
}
\draw (90:\radius) node[above=3pt] {$5$};
\draw (162:\radius) node[above left] {$4$};
\draw (234:\radius) node[below left] {$3$};
\draw (306:\radius) node[below right] {$2$};
\draw (18:\radius) node[above right] {$1$};
\end{tikzpicture}
\end{center}

The completion $\widehat{B}$ of $B$ coincides with the quotient of the completed path
algebra of the graph $C$, i.e.\ the doubled quiver as above,
by the closure of the ideal generated by the relations above.
The algebras $B$ and $\widehat{B}$ were introduced
in \cite{JKS}, Section 3.

The centre $Z$ of $B$ is the polynomial ring $\mathbb{F}[t]$,
where $t=\sum_{i=1}^n x_iy_i$.
The (maximal) Cohen-Macaulay $B$-modules are precisely those which are
free as $Z$-modules. Indeed, such a module $M$ is given by a representation
$\{M_i\,:\,i\in C_0\}$ of
the quiver with each $M_i$ a free $Z$-module of the same rank
(which is the rank of $M$, cf. \cite{JKS}, Section 3).

\begin{de}[\cite{JKS}, Definition 3.5]
For any $B$-module $M$, if $K$ is the field of fractions of $Z$, we can define its \emph{rank}
\[
 {\rm rk}(M) = {\rm len}\bigl( M\otimes_Z K \bigr), 
\]
noting that $B\otimes_Z K\cong M_n ( K)$, 
which is a simple algebra.
\end{de}

It is easy to check that the rank is additive on short exact sequences,
that ${\rm rk} (M) = 0$ for any finite-dimensional $B$-module 
(because these are torsion over $Z$) and 
that, for any Cohen-Macaulay $B$-module $M$ and every idempotent $e_j$, $1\leq j\leq n$,
\[
  {\rm rk}_Z(e_j M) = {\rm rk}(M),
\]
so that, in particular, ${\rm rk}_Z(M) = n  {\rm rk}(M)$.

\begin{de}[\cite{JKS}, Definition 5.1] \label{d:moduleMI}
For any $k$-subset   $I$  of $C_1$, we define a rank $1$ $B$-module
\[
  L_I = (U_i,\ i\in C_0 \,;\, x_i,y_i,\, i\in C_1)
\]
as follows.
For each vertex $i\in C_0$, set $U_i=\mathbb{F}[t]$ and,
for each edge $i\in C_1$, set
\begin{itemize}
\item[] $x_i\colon U_{i-1}\to U_{i}$ to be multiplication by $1$ if $i\in I$, and by $t$ if $i\not\in I$,
\item[] $y_i\colon U_{i}\to U_{i-1}$ to be multiplication by $t$ if $i\in I$, and by $1$ if $i\not\in I$.
\end{itemize}
\end{de}

The module $L_I$ can be represented by a lattice diagram
$\mathcal{L}_I$ in which $U_0,U_1,U_2,\ldots, U_n$ are represented by columns from
left to right (with $U_0$ and $U_n$ to be identified).
The vertices in each column correspond to the natural monomial basis of $\mathbb{F} [t]$.
The column corresponding to $U_{i+1}$ is displaced half a step vertically
downwards (respectively, upwards) in relation to $U_i$ if $i+1\in I$
(respectively, $i+1\not \in I$), and the actions of $x_i$ and $y_i$ are
shown as diagonal arrows. Note that the $k$-subset $I$ can then be read off as
the set of labels on the arrows pointing down to the right which are exposed
to the top of the diagram. For example, the lattice picture $\mathcal{L}_{\{1,4,5\}}$
in the case $k=3$, $n=8$, is shown in the following picture   

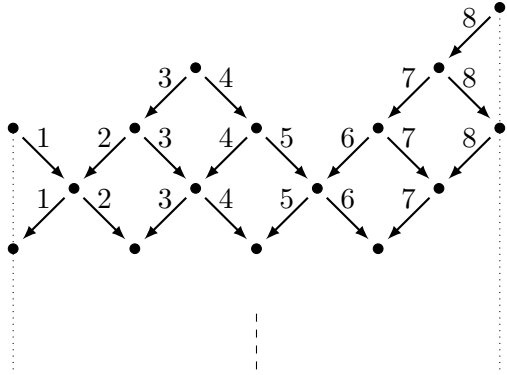
\[
\begin{tikzpicture}[scale=0.8,baseline=(bb.base),
quivarrow/.style={black, -latex, thick}]
\newcommand{\seventh}{51.4} 
\newcommand{\circradius}{1.5cm}
\newcommand{\inradius}{1.2cm}
\newcommand{\outradius}{1.8cm}
\newcommand{\dotrad}{0.1cm} 
\newcommand{\bdrydotrad}{{0.8*\dotrad}} 
\path (0,0) node (bb) {}; 


\draw (0,0) circle(\bdrydotrad) [fill=black];
\draw (0,2) circle(\bdrydotrad) [fill=black];
\draw (1,1) circle(\bdrydotrad) [fill=black];
\draw (2,0) circle(\bdrydotrad) [fill=black];
\draw (2,2) circle(\bdrydotrad) [fill=black];
\draw (3,1) circle(\bdrydotrad) [fill=black];
\draw (3,3) circle(\bdrydotrad) [fill=black];
\draw (4,0) circle(\bdrydotrad) [fill=black];
\draw (4,2) circle(\bdrydotrad) [fill=black];
\draw (5,1) circle(\bdrydotrad) [fill=black];
\draw (6,0) circle(\bdrydotrad) [fill=black];
\draw (6,2) circle(\bdrydotrad) [fill=black];
\draw (7,1) circle(\bdrydotrad) [fill=black];
\draw (7,3) circle(\bdrydotrad) [fill=black];
\draw (8,2) circle(\bdrydotrad) [fill=black];
\draw (8,4) circle(\bdrydotrad) [fill=black];


\draw [quivarrow,shorten <=5pt, shorten >=5pt] (0,2)
-- node[above]{$1$} (1,1);
\draw [quivarrow,shorten <=5pt, shorten >=5pt] (1,1) -- node[above]{$1$} (0,0);
\draw [quivarrow,shorten <=5pt, shorten >=5pt] (2,2) -- node[above]{$2$} (1,1);
\draw [quivarrow,shorten <=5pt, shorten >=5pt] (1,1) -- node[above]{$2$} (2,0);
\draw [quivarrow,shorten <=5pt, shorten >=5pt] (3,3) -- node[above]{$3$} (2,2);
\draw [quivarrow,shorten <=5pt, shorten >=5pt] (2,2) -- node[above]{$3$} (3,1);
\draw [quivarrow,shorten <=5pt, shorten >=5pt] (3,1) -- node[above]{$3$} (2,0);
\draw [quivarrow,shorten <=5pt, shorten >=5pt] (3,3) -- node[above]{$4$} (4,2);
\draw [quivarrow,shorten <=5pt, shorten >=5pt] (4,2) -- node[above]{$4$} (3,1);
\draw [quivarrow,shorten <=5pt, shorten >=5pt] (3,1) -- node[above]{$4$} (4,0);
\draw [quivarrow,shorten <=5pt, shorten >=5pt] (4,2) -- node[above]{$5$} (5,1);
\draw [quivarrow,shorten <=5pt, shorten >=5pt] (5,1) -- node[above]{$5$} (4,0);
\draw [quivarrow,shorten <=5pt, shorten >=5pt] (6,2) -- node[above]{$6$} (5,1);
\draw [quivarrow,shorten <=5pt, shorten >=5pt] (5,1) -- node[above]{$6$} (6,0);
\draw [quivarrow,shorten <=5pt, shorten >=5pt] (6,2) -- node[above]{$7$} (7,1);
\draw [quivarrow,shorten <=5pt, shorten >=5pt] (7,1) -- node[above]{$7$} (6,0);
\draw [quivarrow,shorten <=5pt, shorten >=5pt] (7,3) -- node[above]{$7$} (6,2);
\draw [quivarrow,shorten <=5pt, shorten >=5pt] (7,3) -- node[above]{$8$} (8,2);
\draw [quivarrow,shorten <=5pt, shorten >=5pt] (8,2) -- node[above]{$8$} (7,1);
\draw [quivarrow,shorten <=5pt, shorten >=5pt] (8,4) -- node[above]{$8$} (7,3);

\draw [dotted] (0,-2) -- (0,2);
\draw [dotted] (8,-2) -- (8,4);

\draw [dashed] (4,-2) -- (4,-1);

\end{tikzpicture}
\]

\vspace{2mm}
We see from the above picture that the module $L_I$ is determined by its upper boundary, that is by its rim (this is why we refer to the $k$-subset $I$ as the rim of $L_I$), which is the following directed graph with the leftmost and rightmost vertices identified:  
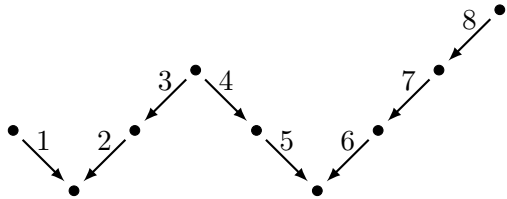
\[
\begin{tikzpicture}[scale=0.8,baseline=(bb.base),
quivarrow/.style={black, -latex, thick}]
\newcommand{\seventh}{51.4} 
\newcommand{\circradius}{1.5cm}
\newcommand{\inradius}{1.2cm}
\newcommand{\outradius}{1.8cm}
\newcommand{\dotrad}{0.1cm} 
\newcommand{\bdrydotrad}{{0.8*\dotrad}} 
\path (0,0) node (bb) {}; 


\draw (0,2) circle(\bdrydotrad) [fill=black];
\draw (1,1) circle(\bdrydotrad) [fill=black];
\draw (2,2) circle(\bdrydotrad) [fill=black];
\draw (3,3) circle(\bdrydotrad) [fill=black];
\draw (4,2) circle(\bdrydotrad) [fill=black];
\draw (5,1) circle(\bdrydotrad) [fill=black];
\draw (6,2) circle(\bdrydotrad) [fill=black];
\draw (7,3) circle(\bdrydotrad) [fill=black];
\draw (8,4) circle(\bdrydotrad) [fill=black];


\draw [quivarrow,shorten <=5pt, shorten >=5pt] (0,2)
-- node[above]{$1$} (1,1);
\draw [quivarrow,shorten <=5pt, shorten >=5pt] (2,2) -- node[above]{$2$} (1,1);
\draw [quivarrow,shorten <=5pt, shorten >=5pt] (3,3) -- node[above]{$3$} (2,2);
\draw [quivarrow,shorten <=5pt, shorten >=5pt] (3,3) -- node[above]{$4$} (4,2);
\draw [quivarrow,shorten <=5pt, shorten >=5pt] (4,2) -- node[above]{$5$} (5,1);
\draw [quivarrow,shorten <=5pt, shorten >=5pt] (6,2) -- node[above]{$6$} (5,1);
\draw [quivarrow,shorten <=5pt, shorten >=5pt] (7,3) -- node[above]{$7$} (6,2);
\draw [quivarrow,shorten <=5pt, shorten >=5pt] (8,4) -- node[above]{$8$} (7,3);



\end{tikzpicture}
\]

Throughout this paper we will identify a rank 1 module $L_I$ with its rim from the above picture. Moreover, most of the time we will omit the arrows in the rim of $L_I$ and represent it as an undirected graph.
\begin{re} {\rm
Note that we represent a rank 1 module $L_I$ by drawing its rim in the plane and identifying the end points of the rim. Unless specified otherwise, we will assume that the leftmost vertex is the vertex labelled by $n$, and in this case, most of the time we will omit labels on the edges of the rim. If one looks at the rim from left to right, then the number of  downward edges in the rim is equal to $k$ (these are the edges labelled by the elements of $I$), and the number of upward edges of the rim is equal to $n-k$ (these are the edges labelled by the elements that don't belong to $I$). }
\end{re}

\begin{prop}[\cite{JKS}, Proposition 5.2]Every rank $1$ Cohen-Macaulay $B$-module is isomorphic to $L_I$
for some unique $k$-subset $I$ of $C_1$.
\end{prop}
Every $B$-module has a canonical endomorphism given by multiplication by $t\in Z$.
For ${L}_I$ this corresponds to shifting $\mathcal{L}_I$ one step downwards.
Since $Z$ is central, ${\rm Hom}_B(M,N)$ is
a $Z$-module for arbitrary $B$-modules $M$ and $N$.
If $M,N$ are free $Z$-modules, then so is ${\rm Hom}_B(M,N)$. In particular, for rank 1 Cohen-Macaulay $B$-modules $L_I$ and $L_J$, ${\rm Hom}_B(L_I,L_J)$ is a free module of rank 1 over $Z=\mathbb F[t]$, generated by the canonical map given by placing the lattice of $L_I$ inside the lattice of $L_J$ as far up as possible so that no part of the rim of $L_I$ is strictly above the rim of $L_J$.

One sees explicitly that
the algebra $B$ has $n$ indecomposable projective left modules $P_j=Be_j$, corresponding to the vertex idempotents $e_j\in B$, for $j\in C_0$.
Our convention is that representations of the quiver correspond to 
left  $B$-modules. Right $B$-modules are representations of the 
opposite quiver. The projective indecomposable $B$-module $P_j$ is the rank 1 module $L_I$, where $I=\{j+1, j+2, \dots, j+k\}$, so we represent projective indecomposable modules as in the following picture, where $P_5$ is pictured ($n=5$, $k=3$):
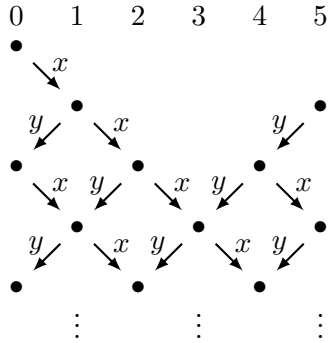
\begin{center}
\begin{tikzpicture}[scale=0.8]
\foreach \j in {0,...,5}
  {\path (\j,3.5) node (a) {$\j$};}
\path (0,3) node (a0) {$\bullet$}; \path (5,2) node (a5) {$\bullet$}; 
\foreach \v/\x/\y in
  {a1/1/2, a2/2/1, a3/3/0, a4/4/1, b0/0/1, b1/1/0, b2/2/-1, b4/4/-1, b5/5/0, c0/0/-1}
  {\path (\x,\y) node (\v) {$\bullet$};}
\foreach \j in {1,3,5}
  {\path (\j,-1.5) node {$\vdots$};}
\foreach \t/\h in
  {a0/a1, a1/a2, a2/a3, b0/b1, b1/b2, a3/b4, a4/b5}
  {\path[->,>=latex] (\t) edge[black,thick] node[black,above right=-2pt] {$x$} (\h);}
\foreach \t/\h in
  {a4/a3, a5/a4, b5/b4, a3/b2, a2/b1, a1/b0, b1/c0}
  {\path[->,>=latex] (\t) edge[black,thick] node[black,above left =-3pt] {$y$} (\h);}
\end{tikzpicture}
\end{center}

\section{Periodicity  of rank 1 modules}

In this section we prove that all rank 1 Cohen-Macaulay $B$-modules are periodic, and we give an explicit formula for the periods of these modules in terms of their rims. 

If $L_I$ is a rank 1 module corresponding to a rim $I$, then the projective cover of $L_I$ is given by 
$$\bigoplus_{u\in U} P_u \xrightarrow{\pi}  L_{I}$$ where $U=\{u  \notin I \vert u+1 \in I  \}$, and $\pi$ is given by the canonical maps from $P_u$ to $L_I$, for every $u\in U$, i.e.\ the maps that map $P_u$ to $L_I$ by placing the rim of $P_u$ inside the $L_I$ as far up as possible so that no parts of the rim of $P_u$ are strictly above the rim of $L_I$. In other words, the projective cover is determined by the projective indecomposable modules that correspond to the peaks of the rim of $L_I$. The rank of the projective cover of $L_I$ is equal to the number of the peaks of the rim of $L_I$. Since the rank is additive on short exact sequences, we have that the rank of the kernel of the projective cover of $L_I$, denote it by $r$, is one less than the number of peaks of the rim, that is, if there are $r+1$ peaks on the rim of $L_I$, then the rank of the first syzygy of $L_I$ is $r$.  

Denote the kernel of $\pi$ by $\Omega(L_I)$. To determine the projective cover of  $\Omega(L_I)$ we look at the following picture, where only parts of the lattice of the module $L_I$ are drawn.  The kernel of $\pi$  corresponds to the parts of the lattice $\mathcal{L}_I$ that are on or  below the dashed line. This area corresponds to the part of $L_I$ that is covered by at least two different projective indecomposable modules $P_u$ from the set $U$. For example, the dot that is singled out below the dashed line in the picture represents an element of $L_I$ that is covered both by  $P_{10}$ and $P_8$.

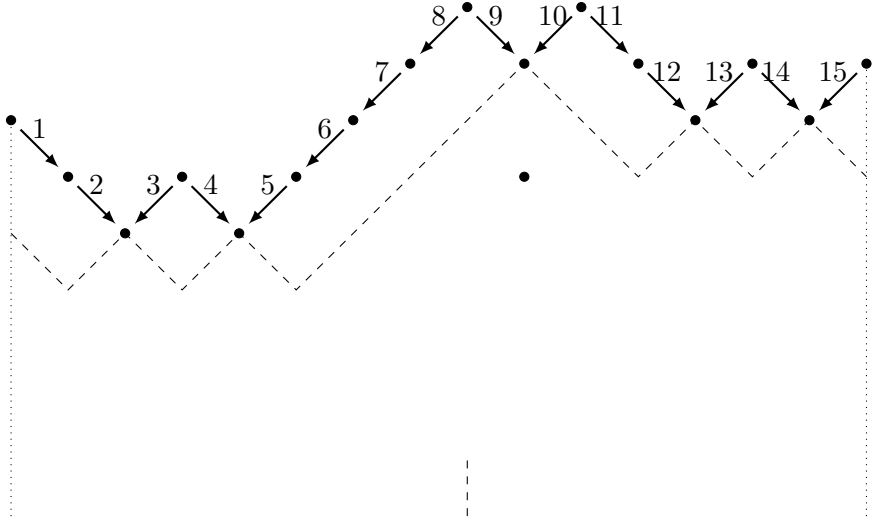
\[
\begin{tikzpicture}[scale=0.75,baseline=(bb.base),
quivarrow/.style={black, -latex, thick}]
\newcommand{\seventh}{51.4} 
\newcommand{\circradius}{1.5cm}
\newcommand{\inradius}{1.2cm}
\newcommand{\outradius}{1.8cm}
\newcommand{\dotrad}{0.1cm} 
\newcommand{\bdrydotrad}{{0.8*\dotrad}} 
\path (0,0) node (bb) {}; 


\draw (0,2) circle(\bdrydotrad) [fill=black];
\draw (1,1) circle(\bdrydotrad) [fill=black];
\draw (2,0) circle(\bdrydotrad) [fill=black];
\draw (3,1) circle(\bdrydotrad) [fill=black];
\draw (4,0) circle(\bdrydotrad) [fill=black];
\draw (5,1) circle(\bdrydotrad) [fill=black];
\draw (6,2) circle(\bdrydotrad) [fill=black];
\draw (7,3) circle(\bdrydotrad) [fill=black];
\draw (8,4) circle(\bdrydotrad) [fill=black];
\draw (9,3) circle(\bdrydotrad) [fill=black];
\draw (10,4) circle(\bdrydotrad) [fill=black];
\draw (11,3) circle(\bdrydotrad) [fill=black];
\draw (12,2) circle(\bdrydotrad) [fill=black];
\draw (13,3) circle(\bdrydotrad) [fill=black];
\draw (14,2) circle(\bdrydotrad) [fill=black];
\draw (15,3) circle(\bdrydotrad) [fill=black];



\draw [quivarrow,shorten <=5pt, shorten >=5pt] (0,2)
-- node[above]{$1$} (1,1);
\draw [quivarrow,shorten <=5pt, shorten >=5pt] (1,1) -- node[above]{$2$} (2,0);
\draw [quivarrow,shorten <=5pt, shorten >=5pt] (3,1) -- node[above]{$3$} (2,0);
\draw [quivarrow,shorten <=5pt, shorten >=5pt] (3,1) -- node[above]{$4$} (4,0);
\draw [quivarrow,shorten <=5pt, shorten >=5pt] (5,1) -- node[above]{$5$} (4,0);
\draw [quivarrow,shorten <=5pt, shorten >=5pt] (6,2) -- node[above]{$6$} (5,1);
\draw [quivarrow,shorten <=5pt, shorten >=5pt] (7,3) -- node[above]{$7$} (6,2);
\draw [quivarrow,shorten <=5pt, shorten >=5pt] (8,4) -- node[above]{$8$} (7,3);
\draw [quivarrow,shorten <=5pt, shorten >=5pt] (8,4) -- node[above]{$9$} (9,3);
\draw [quivarrow,shorten <=5pt, shorten >=5pt] (10,4) -- node[above]{$10$} (9,3);
\draw [quivarrow,shorten <=5pt, shorten >=5pt] (10,4) -- node[above]{$11$} (11,3);
\draw [quivarrow,shorten <=5pt, shorten >=5pt] (11,3) -- node[above]{$12$} (12,2);
\draw [quivarrow,shorten <=5pt, shorten >=5pt] (13,3) -- node[above]{$13\,\,$} (12,2);
\draw [quivarrow,shorten <=5pt, shorten >=5pt] (13,3) -- node[above]{$14\,\,$} (14,2);
\draw [quivarrow,shorten <=5pt, shorten >=5pt] (15,3) --node[above]{$15\,\,$}(14,2);


\draw[dashed] (0,0)--(1,-1)--(2,0)--(3,-1)--(4,0)--(5,-1)--(9,3)--(11,1)--(12,2)--(13,1)--(14,2)--(15,1);

\draw (9,1) circle(\bdrydotrad) [fill=black]; 



\draw [dotted] (0,-5) -- (0,2);
\draw [dotted] (15,-5) -- (15,3);

\draw [dashed] (8,-5) -- (8,-4);

\end{tikzpicture}
\]
\vspace{3mm}

It follows that the projective cover of $\Omega(L_I)$ is determined by the low points of the rim of $L_I$. We call these points the valleys of the rim $I$.  It is clear that there are as many low points on the rim as there are peaks. Hence, the projective cover of $\Omega(L_I)$ is a module of rank $r+1$ isomorphic to the direct sum 
$$\bigoplus_{v\in V} P_v,$$ where $V=\{v  \in I \vert v+1 \notin I  \}$.

Again, because the rank is additive on short exact sequences, it follows that the kernel of the projective cover of $\Omega(L_I)$ is a rank 1 module. This means that there is a $k$-subset of  $\{1,2,\dots,n\}$, denoted by $I^2$,  such that this kernel (the second syzygy of $L_I$), denoted by $\Omega^2(L_I)$, is isomorphic to $L_{I^2}$, i.e.\  $\Omega^2(L_I)\cong L_{I^2}$. Using the same arguments, the projective cover of $\Omega^2(L_I)\cong L_{I^2}$ is a module of rank $r+1$, and the kernel of this projective cover, denoted by $\Omega^3(L_I)$, is a rank $r$ module, and the kernel of the projective cover of $\Omega^3(L_I)$ is a rank $1$ module, denoted by $\Omega^4(L_I)$, and it is isomorphic to $L_{I^4}$  for  some $k$-subset $I^4$ of $\{1,2,\dots,n\}$. If we continue this construction of the minimal projective resolution of $L_I$, every other kernel will be a rank 1 module. 

Since there are only finitely many rank 1 modules (they are in  bijection with $k$-subsets of $\{1,2,\dots,n\}$), we must have that the projective resolution of $L_I$ is periodic. That is, for some indices $a$ and $b$, $a\neq b$, it holds that $\Omega^a(L_{I})\cong \Omega^b(L_{I})$, with $\Omega^a(L_{I})$ denoting the $a$th syzygy of $L_I$. In fact, we are going to prove a stronger statement that for some index $t$, we have that $\Omega^{t}(L_I)\cong L_I$. The rest of this section is devoted to determining the minimal such an index $t$. 

Obviously, when $\Omega^1(L_I)$ is of rank greater than 1, $t$ must be an even number. Thus,  we have to consider separately the case when $\Omega^1(L_I)$ is a rank 1 module, because in this case in each step of the minimal projective resolution we get kernels that are rank 1 modules, so it can happen that in an odd number of steps we get a kernel that is isomorphic to $L_I$, as we will see in the upcoming example.

\begin{ex}\label{ex1}
{\rm  Let  $n=6, k=4,$ and $I=\{1,2,4,5\}$. In this case, the number of peaks on the rim of $L_I$ is equal to 2. For every $i$, $\Omega^i({L_I})$ is a rank 1 module. 

The rims of the rank 1 modules $\Omega^i({L_I})$, for $i=1,2,3$, are depicted with different types of lines in the following picture, with the dashed rim representing the rim of $\Omega^1({L_I})$, the thin lined rim representing the rim of $\Omega^2({L_I})$, and the dotted rim representing the rim of $\Omega^3({L_I})$. We see from the picture  that $\Omega^3(L_I)\cong L_I$, and that the period  of $L_I$ is 3.} 

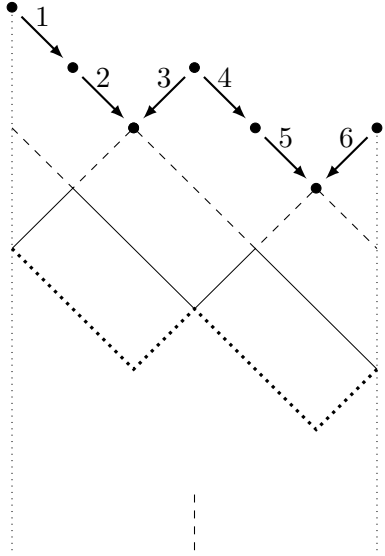
\[
\begin{tikzpicture}[scale=0.8,baseline=(bb.base),
quivarrow/.style={black, -latex, thick}]
\newcommand{\seventh}{51.4} 
\newcommand{\circradius}{1.5cm}
\newcommand{\inradius}{1.2cm}
\newcommand{\outradius}{1.8cm}
\newcommand{\dotrad}{0.1cm} 
\newcommand{\bdrydotrad}{{0.8*\dotrad}} 
\path (0,0) node (bb) {}; 


\draw (0,5) circle(\bdrydotrad) [fill=black];
\draw (1,4) circle(\bdrydotrad) [fill=black];
\draw (2,3) circle(\bdrydotrad) [fill=black];
\draw (3,4) circle(\bdrydotrad) [fill=black];
\draw (4,3) circle(\bdrydotrad) [fill=black];
\draw (5,2) circle(\bdrydotrad) [fill=black];
\draw (6,3) circle(\bdrydotrad) [fill=black];

\draw (2,3) circle(\bdrydotrad) [fill=black];
\draw[dashed] (0,3)--(1,2)--(2,3)--(3,2)--(4,1)--(5,2)--(6,1); 
\draw[-, very thin] (0,1)--(1,2)--(2,1)--(3,0)--(4,1)--(5,0)--(6,-1); 
\draw[very thick, dotted] (0,1)--(1,0)--(2,-1)--(3,0)--(4,-1)--(5,-2)--(6,-1); 




\draw [quivarrow,shorten <=5pt, shorten >=5pt] (0,5)
-- node[above]{$1$} (1,4);
\draw [quivarrow,shorten <=5pt, shorten >=5pt] (1,4) -- node[above]{$2$} (2,3);
\draw [quivarrow,shorten <=5pt, shorten >=5pt] (3,4) -- node[above]{$3$} (2,3);
\draw [quivarrow,shorten <=5pt, shorten >=5pt] (3,4) -- node[above]{$4$} (4,3);
\draw [quivarrow,shorten <=5pt, shorten >=5pt] (4,3) -- node[above]{$5$} (5,2);
\draw [quivarrow,shorten <=5pt, shorten >=5pt] (6,3) -- node[above]{$6$} (5,2);





\draw [dotted] (0,-4) -- (0,5);
\draw [dotted] (6,-4) -- (6,3);

\draw [dashed] (3,-4) -- (3,-3);

\end{tikzpicture}
\]
\vspace{2mm} 
\end{ex}

Before moving on to the general case when $\Omega(L_I)$ is a module of rank 1, let us introduce some of the notation used in this section. 

If $I$ is a $k$-subset of $\{1,2,\dots,n\}$ such that the kernel of the projective cover of  $L_I$ is a rank $r$ module, then $I$ can be written as a disjoint union of $r+1$ segments $A_1, A_2, \dots, A_{r+1}$, where $A_i=[a_i,b_i]$, and $a_{i+1}-b_i>1$, for all $i$. We can also assume without loss of generality that $a_1=1$, because we can always assume that 0 is one of the peaks of the rim $I$, by renumbering if necessary. The size of the segment $A_i$ is denoted by $d_i$, and the difference $a_{i+1}-b_i-1$ is denoted by $l_i$. If one considers the rim of the module $L_I$, it is clear that the numbers $d_i$ (respectively \ $l_i$) represent the sizes of downward slopes (respectively upward slopes) of the rim, when looked at from left to right. Also, $\sum d_i=k$, and $\sum l_i=n-k$.

\begin{ex}
{\rm
Continuing the previous example, we have that $I$ is the union   $I=\{1,2\}\cup \{4,5\}$, and  $r+1=2$. There are two downward slopes, both of length 2, i.e.\ $d_1=d_2=2$, and there are two upward slopes, both of length 1, i.e.\ $l_1=l_2=1.$}
\end{ex}

\subsection{Kernels of rank 1}

We will start by dealing with the case when ${\rm rk}\, \Omega(L_I)=1$, i.e.\ the case when we have only two peaks on the rim of $L_I$. In this situation, there are positive integers $d_1,d_2,l_1, l_2$, such that $$I=A_1\cup A_2=\{1,2,\dots,d_1\}\cup \{d_1+l_1+1,d_1+l_1+2,\dots,d_1+l_1+d_2\}.$$  

A part of the lattice of $L_I$ is drawn in the following picture (note that the actual lengths of the downward and upward slopes of the rim of $L_I$ are $d_1$ and $d_2$ for the downward slopes, and $l_1$ and $l_2$ for the upward slopes). 

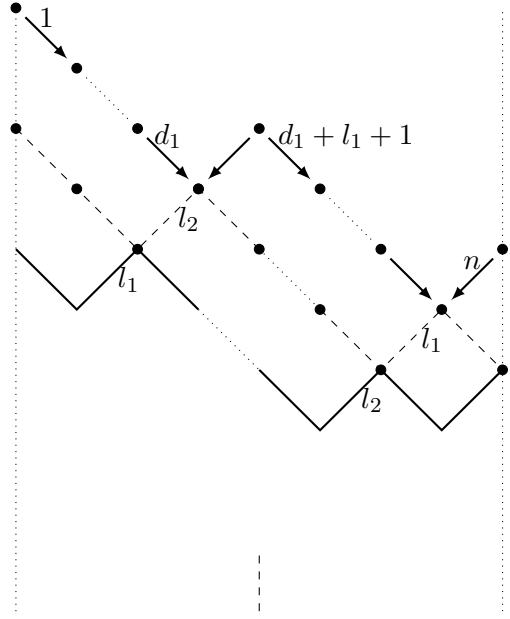
\[
\begin{tikzpicture}[scale=0.8,baseline=(bb.base),
quivarrow/.style={black, -latex, thick}]
\newcommand{\seventh}{51.4} 
\newcommand{\circradius}{1.5cm}
\newcommand{\inradius}{1.2cm}
\newcommand{\outradius}{1.8cm}
\newcommand{\dotrad}{0.1cm} 
\newcommand{\bdrydotrad}{{0.8*\dotrad}} 
\path (0,0) node (bb) {}; 


\draw (-1,6) circle(\bdrydotrad) [fill=black];
\draw (-1,4) circle(\bdrydotrad) [fill=black];

\draw (0,5) circle(\bdrydotrad) [fill=black];
\draw (1,4) circle(\bdrydotrad) [fill=black];
\draw (2,3) circle(\bdrydotrad) [fill=black];
\draw (3,4) circle(\bdrydotrad) [fill=black];
\draw (4,3) circle(\bdrydotrad) [fill=black];
\draw (5,2) circle(\bdrydotrad) [fill=black];

\draw (0,3) circle(\bdrydotrad) [fill=black];
\draw (1,2) circle(\bdrydotrad) [fill=black];
\draw (2,3) circle(\bdrydotrad) [fill=black];
\draw (3,2) circle(\bdrydotrad) [fill=black];
\draw (4,1) circle(\bdrydotrad) [fill=black];
\draw (6,1) circle(\bdrydotrad) [fill=black];
\draw (7,2) circle(\bdrydotrad) [fill=black];
\draw (5,0) circle(\bdrydotrad) [fill=black];
\draw (7,0) circle(\bdrydotrad) [fill=black];

\draw[dashed] (-1,4)--(0,3)--(1,2)--node[right]{$l_2\,\,\,\,\,\,\,\,\,\,\,\,\,\,\, $}(2,3)--(3,2);
\draw[dotted] (3,2)--(4,1);
\draw[dashed](4,1)--(5,0)--node[right]{$l_1$}(6,1)--(7,0); 
\draw[-,  thick] (-1,2)--(0,1)--node[right]{$l_1\,\,\,\,\,\,\,\,\,\,\,\,\,\,\, $}(1,2)--(2,1);
\draw[dotted](2,1)--(3,0);
\draw[-,  thick] (3,0)--(4,-1)--node[right]{$l_2\,\,\,\,\,\,\,\,\,\,\,\,\,\,\, $}(5,0)--(6,-1)--(7,0); 



\draw [quivarrow,shorten <=5pt, shorten >=5pt] (-1,6) -- node[above]{$1$} (0,5);
\draw [dotted,shorten <=5pt, shorten >=5pt] (0,5)
-- node[above]{} (1,4);
\draw [quivarrow,shorten <=5pt, shorten >=5pt] (1,4) -- node[above]{$d_1$} (2,3);
\draw [quivarrow,shorten <=5pt, shorten >=5pt] (3,4) -- node[above]{} (2,3);
\draw [quivarrow,shorten <=5pt, shorten >=5pt] (3,4) -- node[above]{$\,\,\,\,\,\,\,\quad\,\,\,\,\,\,\,\,\,\,d_1+l_1+1$} (4,3);
\draw [shorten <=5pt, shorten >=5pt,dotted] (4,3) -- node[above]{} (5,2);
\draw [quivarrow,shorten <=5pt, shorten >=5pt] (7,2) -- node[above]{$n$} (6,1);
\draw [quivarrow,shorten <=5pt, shorten >=5pt] (5,2) -- node[above]{$$} (6,1);





\draw [dotted] (-1,-4) -- (-1,6);
\draw [dotted] (7,-4) -- (7,6);

\draw [dashed] (3,-4) -- (3,-3);

\end{tikzpicture}
\]

The projective cover of $L_I$ is $P_0\oplus P_{d_1+l_1}$. The kernel of the projective cover is a rank 1 module whose rim is given by reading off its peaks from the rim of $L_I$, that is, by reading off the valleys of the rim of $L_I$, and it is depicted by the dashed line in the above picture.  By looking at the above picture we see that the rim of $\Omega^1(L_I)$ has its peaks at $d_1$ and $d_1+l_1+d_2$. Thus, $\Omega^1(L_I)\cong L_{I^1}$, where $I^1=\{ 1-l_2, 2-l_2, \dots, d_1-l_2  \} \cup \{ d_1+1, d_1+2,\dots, d_1+d_2\}$, with the addition being modulo $n$. The rim of $\Omega^1(L_I)$, drawn by the dashed line in the above picture, is obtained from the rim of $L_I$ by taking for its peaks the valleys of the rim of $L_I$, and by shifting the upward slopes of the rim of $L_I$ to the right, meaning that the upward slope that started at the $i$th valley  (reading from left to right) of the rim of $L_I$ now starts at the end of the $i+1$th downward slope in the rim of $\Omega(L_I)$, as in the above picture.  We obtained that $I^1=A_1^1\cup A_2^1=\{ 1-l_2, 2-l_2, \dots, d_1-l_2  \} \cup \{ d_1+1, d_1+2,\dots, d_1+d_2\}$ and that the gap (the length of the upward slope) between $A_1^1$, which is a set of size $d_1$, and $A_2^1$, which is a set of size $d_2$, is $l_2$. 

If we now compute the projective cover of $L_{I^1}$, by using the same arguments we get that the kernel of this projective cover, $\Omega^2(L_{I})$,  is isomorphic to $L_{I^2}$, where $I^2=A_1^2\cup A_2^2=\{ 1-l_2-l_1, 2-l_2-l_1, \dots, d_1-l_2-l_1  \} \cup \{ d_1+1-l_2, d_1+2-l_2,\dots, d_1+d_2-l_2\}$. The rim of $L_{I^2}$ is drawn by the thick line in the above picture. Using that $l_1+l_2=n-k$ and adding modulo $n$, we get that $I^2=\{ 1+k, 2+k, \dots, d_1+k  \} \cup \{ d_1+l_1+1+k, d_1+l_1+2+k,\dots, d_1+d_2+l_1+k\}$, and the gap between $A_1^2$ and $A_2^2$ is $l_1$. Repeating this procedure we get an explicit description of the kernels appearing in the minimal projective resolution of $L_I$, i.e.\ we get $\Omega^m(L_I)\cong L_{I^m}$, where $I^m=A_1^m\cup A_2^m $. After even number of steps $2t$, we get that 
$\Omega^{2t}(L_{I})\cong L_{I^{2t}}$, where $$I^{2t}=\{ 1+tk, 2+tk, \dots, d_1+tk  \} \cup \{ d_1+l_1+1+tk, \dots, d_1+d_2+l_1+tk\},$$ and the gap between $A_1^{2t}$ and $A_2^{2t}$ is $l_1$.  After odd number of steps we get that 
$\Omega^{2t+1}(L_{I})\cong L_{I^{2t+1}}$, where 
 $$I^{2t+1}= \{ 1-l_2+tk, 2-l_2+tk, \dots, d_1-l_2+tk  \} \cup$$ $$ \{ d_1+1+tk, d_1+2+tk,\dots, d_1+d_2+tk\},$$ and the gap between $A_1^{2t+1}$ and $A_2^{2t+1}$ is $l_2$.

\begin{te}
Let $L_I$ be a rank $1$ module whose kernel of its projective cover is a module of rank $1$, and let $\Omega^m(L_I)$ be as above. It holds that $L_I\cong \Omega^{2n/(n,k)}(L_I)$. The minimal projective resolution of $L_I$ is periodic with period dividing $2n/(n,k)$. 
\end{te}
\begin{proof}
Keeping the notation from the above discussion, if we set $t=n/(n,k)$, then  $A_1=A_1^{2t},\, A_2=A_2^{2t}$, i.e.\ $I=I^{2t}$.  This means that $L_I\cong \Omega^{2n/(n,k)}(L_I)$.  
\end{proof}

We will now proceed by giving the explicit formula for the period of a rank 1 module with the kernel of the projective cover of rank 1 as well.  We are looking for a minimal index $m$ such that $I^m=I$. 

If $d_1\neq d_2$ and $l_1\neq l_2$, then $m$
has to be an even number in order for the upward slopes to be in the correct order. The condition $A_1=A_1^m$ is equivalent to the condition $km/2 \equiv 0 \mod  n$ which we get from the requirement that the smallest elements of $A_1$ and $A_1^m$ are equal.  Hence, in this case $m=2t$, where $t$ is the minimal positive integer such that $kt \equiv 0\mod  n$, i.e.\ 
\begin{equation}\label{case1} m=2 n/(n,k),\end{equation} with $(n,k)$ being the greatest common divisor of $n$ and $k$. If $(n,k)=1$, then  we obtain $2n$, which is the upper bound from the previous theorem.  

If $d_1=d_2$ and $l_1\neq l_2$, then, in the general case, $m$ could either be even or odd. If $m$ is even, then as in the previous case it is equal to  $2 n/(n,k)$. If $m=2t+1$ is odd, then the gap between $A_1^m$ and $A_2^m$ is $l_2$, forcing that  $A_2^m=A_1$ and $A_1^m=A_2$. This is equivalent to saying that $d_1+1+tk\equiv 1 \mod n$, hence $m=2t+1$, where $t$ is the minimal positive integer such that  $d_1+tk\equiv 0 \mod n$. Therefore, in this case 
\begin{equation}\label{case2} m=\min\{   2 n/(n,k),   2\min\{t\,\vert\,  d_1+tk\equiv 0 \!\!\!\!\mod n\}  +1\}. \end{equation} 

If $d_1\neq d_2$ and $l_1=l_2$, then, in the general case, $m$ could either be even or odd, since the gaps between $A_1^i$ and $A_2^i$ are in the right order for every $i$. In this case it must be $A_1^m=A_1$. This condition is equivalent to  the condition  $1\equiv 1+tk \mod n$ when $m=2t$ is even, and  $1+tk-l_2 \equiv 1 $ when $m=2t+1$ is odd. Therefore, in this case 
\begin{equation}\label{case3} m=\min\{   2 n/(n,k),  2 \min\{t\,\vert\,  tk-l_2 \equiv 0\!\!\!\! \mod n\} +1 \}. \end{equation} 

We are left with the most complicated case when $d_1=d_2$ and $l_1=l_2$. Again, $m$ could be either even or odd, but also, it can either be that $A_1=A_1^m$ and $A_1=A_2^m$, because the gaps will be in the right order,  and $A_1^m$ could be each of the sets $A_1$ and $A_2$.  If $A_1=A_1^m$, then  $1+tk-l_2 \equiv 1\!\! \mod n$ when $m=2t+1$ is odd, and $1+tk \equiv 1\!\! \mod n$ when $m=2t $   is even. If  $A_2=A_1^m$, then   $d_1+1+tk \equiv 1 \!\!\mod n$ when $m=2t+1$ is odd, and $1+d_1+l_1+tk \equiv 1 \!\!\mod n$ when $m=2t $   is even. Hence, in this case $m$ is a divisor of $2n/(n,k)$ given by 
\begin{equation}\label{case4}\min \left\{    \begin{array}{l}  
2 n/(n,k),\\ 
 2 \min\{t\,\vert \, tk-l_2 \equiv 0\!\! \mod n\} +1, \\
 2\min\{t\,\vert \, d_1+tk\equiv 0\!\! \mod n\}  +1, \\
  2\min\{t\,\vert \, d_1+l_1+tk\equiv 0\!\! \mod n\}.   \end{array} \right. \end{equation}

We summarize our results in the following theorem. 

\begin{te}
Let $L_I$ be a rank $1$ Cohen-Macaulay module, and let $d_1,d_2$ and $l_1,l_2$ be as above.  Depending on whether $d_1=d_2$ or not, and $l_1=l_2$ or not, the period of the module $L_I$ is given by equations $(\ref{case1})$, $(\ref{case2})$, $(\ref{case3})$ and $(\ref{case4})$. 
\end{te}

This completes our determination of the periods for the rank 1 Cohen-Macaulay modules whose kernels of the projective cover are also rank 1 modules. For four different cases studied above, in general, we have four different formulas for computation of the period of a given rank 1 module. 

\begin{ex}

{\rm In the Example \ref{ex1} we had $n=6, k=4$ and a rank 1 Cohen-Macaulay module $L_I$ with the rim $ I=\{1,2,4,5\}$, and $d_1=d_2=2$ and $l_1=l_2=1$. In this case the period of $L_I$ is given by equation (\ref{case4}). For $t=1$, we have that $d_1+kt\equiv 0 \!\mod 6$, meaning that the period of the module $L_I$ is 3. }

\end{ex}

\begin{ex} {\rm 
Let  $n=6, k=3$ and $I=\{1,2,5\}$. In this case we have that $d_1=2\neq d_2=1$, $l_1=2\neq l_2=1$. Since $k=3$, it follows that the period of $L_I$ is $m=2n/(n,k)=4$. The rims of $\Omega^i({L_I})$ are depicted in different types of lines in the following picture, with thick dashed line representing the rim of $\Omega^4(L_I)$, which is isomorphic to $L_I$.}

\[
\begin{tikzpicture}[scale=0.8,baseline=(bb.base),
quivarrow/.style={black, -latex, thick}]
\newcommand{\seventh}{51.4} 
\newcommand{\circradius}{1.5cm}
\newcommand{\inradius}{1.2cm}
\newcommand{\outradius}{1.8cm}
\newcommand{\dotrad}{0.1cm} 
\newcommand{\bdrydotrad}{{0.8*\dotrad}} 
\path (0,0) node (bb) {}; 


\draw (0,5) circle(\bdrydotrad) [fill=black];
\draw (4,5) circle(\bdrydotrad) [fill=black];
\draw (5,4) circle(\bdrydotrad) [fill=black];
\draw (6,5) circle(\bdrydotrad) [fill=black];

\draw (1,4) circle(\bdrydotrad) [fill=black];
\draw (2,3) circle(\bdrydotrad) [fill=black];
\draw (3,4) circle(\bdrydotrad) [fill=black];
\draw (2,1) circle(\bdrydotrad) [fill=black];
\draw (0,1) circle(\bdrydotrad) [fill=black];
\draw (6,3) circle(\bdrydotrad) [fill=black];
\draw (6,-1) circle(\bdrydotrad) [fill=black];
\draw (4,-1) circle(\bdrydotrad) [fill=black];
\draw (0,-1) circle(\bdrydotrad) [fill=black];
\draw (5,-2) circle(\bdrydotrad) [fill=black];
\draw (2,-3) circle(\bdrydotrad) [fill=black];

\draw (0,3) circle(\bdrydotrad) [fill=black];
\draw (1,2) circle(\bdrydotrad) [fill=black];
\draw (2,3) circle(\bdrydotrad) [fill=black];
\draw (3,2) circle(\bdrydotrad) [fill=black];
\draw (5,0) circle(\bdrydotrad) [fill=black];
\draw (6,1) circle(\bdrydotrad) [fill=black];
\draw[dashed] (0,3)--(1,2)--(2,3)--(3,2)--(4,3)--(5,4)--(6,3); 
\draw[dotted, thick] (0,1)--(1,2)--(2,1)--(3,2)--(4,1)--(5,0)--(6,1); 
\draw[thick] (0,-1)--(2,1)--(3,0)--(4,-1)--(5,0)--(6,-1); 
\draw[dashed, very thick] (0,-1)--(2,-3)--(4,-1)--(5,-2)--(6,-1);




\draw [quivarrow,shorten <=5pt, shorten >=5pt] (0,5)
-- node[above]{$1$} (1,4);
\draw [quivarrow,shorten <=5pt, shorten >=5pt] (1,4) -- node[above]{$2$} (2,3);
\draw [quivarrow,shorten <=5pt, shorten >=5pt] (3,4) -- node[above]{$3$} (2,3);
\draw [quivarrow,shorten <=5pt, shorten >=5pt] (4,5) -- node[above]{$5$} (5,4);
\draw [quivarrow,shorten <=5pt, shorten >=5pt] (4,5) -- node[above]{$4$} (3,4);
\draw [quivarrow,shorten <=5pt, shorten >=5pt] (6,5) -- node[above]{$6$} (5,4);





\draw [dotted] (0,-4) -- (0,5);
\draw [dotted] (6,-4) -- (6,5);

\draw [dashed] (3,-4) -- (3,-3);

\end{tikzpicture}
\]

\end{ex}

\subsection{Kernels  of rank greater than 1}

We now assume that $I$ is such that the kernel of the projective cover of $L_I$ is of rank greater than 1, i.e.\ the rim of $L_I$ has three or more peaks, and we set ${\rm rk}\, \Omega(L_I)=r>1.$ 

From the above discussion we have that every other kernel in the projective resolution of $L_I$ is a rank 1 module. If $I$ is a disjoint union of segments $A_1, A_2,\dots, A_{r+1}$, then we assume that $A_i$ has $d_i$ elements and that the gap between $A_{i}$ and $A_{i+1}$ is of size $l_i$. Also, we  can assume without loss of generality that the smallest element in $A_1$ is 1, i.e.\  $A_1=\{1,2,\dots, d_1\}$, $A_2=\{d_1+l_1+1,\dots, d_1+l_1+d_2\}, \dots , $  $A_{r+1}=\{ \sum_{i=1}^r d_i+\sum_{i=1}^r l_i+1,\dots, \sum_{i=1}^r d_i+\sum_{i=1}^r l_i+d_{r+1} \}$.

A projective resolution of $L_{I}$ is 
\[
\bigoplus_{v\in V} P_v \xrightarrow{D} \bigoplus_{u\in U} P_u \to  L_{I} \to 0,
\]
where $U=\{ u\not\in I : u+1 \in I\}$ and $V=\{ v \in I : v+1 \not\in I\}$.
Note that $U$ and $V$ are disjoint sets with the same number of elements, 
which alternate in the cyclic order.
This number is $r+1$, where $r={\rm rk}\, \Omega (L_{I})$
and $\Omega (L_{I})={\rm im}\, D$ is the first syzygy.
The $(r+1)\times(r+1)$ matrix $D=(d_{vu})$ has only non-zero entries
when $u,v$ are adjacent in $U\cup V$.
More precisely,
\begin{equation}\label{eq:Dcoeffs}
 d_{vu} = \begin{cases}
 x^{v-u} & \text{when $u$ precedes $v$,}\\
 -y^{u-v} & \text{when $u$ follows $v$,}\\
  0 & \text{otherwise.}
 \end{cases}
\end{equation}
Here, $x$ and $y$ should be interpreted as $x_i$ and $y_j$ for appropriate indices $i$ and $j$.

Thus, we can assume that the matrix $D$ is supported on just two cyclic 
diagonals. Hence, it is of the following form (with omitted entries all equal to zero):
\[
 \begin{pmatrix}
 \bullet &&& \bullet  \\
 \bullet & \bullet && \\
 & \ddots & \ddots &  \\
 & & \bullet & \bullet
 \end{pmatrix}.
\]

Note that the lower cyclic diagonal contains the top right entry of the matrix.

\subsubsection{The kernel of the projective cover $D$} We proceed by computing the kernel of the above mentioned map $D$ from the projective resolution of $L_I$. We know that this kernel is a rank 1 module.  If $I=\{a_1,a_2,\dots,a_h\}$, then we set  $I+k=\{a_1+k,a_2+k,\dots,a_h+k\}$.

\begin{prop} The rim of the second syzygy of $L_I$ is the rim $I$ shifted by $k$, that is, the rim of $\Omega^2(L_I)$ is $I+k$.
\end{prop}
\begin{proof} 
If we fix a valley $v$ of the rim $I$, then the elements of  the module $P_v$, where $P_v$ is a summand of  $\bigoplus_{v\in V} P_v$, are mapped by the map $D$ to two projective modules $P_{u_{vl}}$ and $P_{u_{vr}}$, where $u_{vl}$ denotes the peak that is to the left of $v$ and $u_{vr}$ denotes the peak that is to the right of the valley $v$. So for a given peak $u$, only two $P_vs$ are mapped to $P_u$. 

For example, if we look at the rim from the following picture, for $P_{10}$, only $P_{9}$ and $P_{12}$ are mapped into $P_{10}$.  So the parts of $P_{12}$ that are potentially in the kernel of $D$ are the ones lying on or below the thick dotted line corresponding to $P_{9}$ in the below picture. Also, $P_{12}$ is mapped into $P_{13}$ so the same has to hold with respect to $P_{14}$, the only parts of $P_{12}$ that are candidates for the kernel are the ones on or below the thin dotted  line corresponding to $P_{14}$, i.e.\ the parts below both thick dotted line and thin dotted line. But not everything below the thin dotted line is a candidate for the kernel since the only legitimate candidates from $P_{14}$ are the ones below the black thin line corresponding to  $P_2$,  and so on. We conclude that we obtain the kernel of $D$ by reading off its rim from the rim of $I$ by taking all the elements that are below the rims of all projective indecomposable modules $P_v$, $v\in V$ (in other words, that belong to the intersection of all projective indecomposable modules), i.e. below all lines in the following picture.

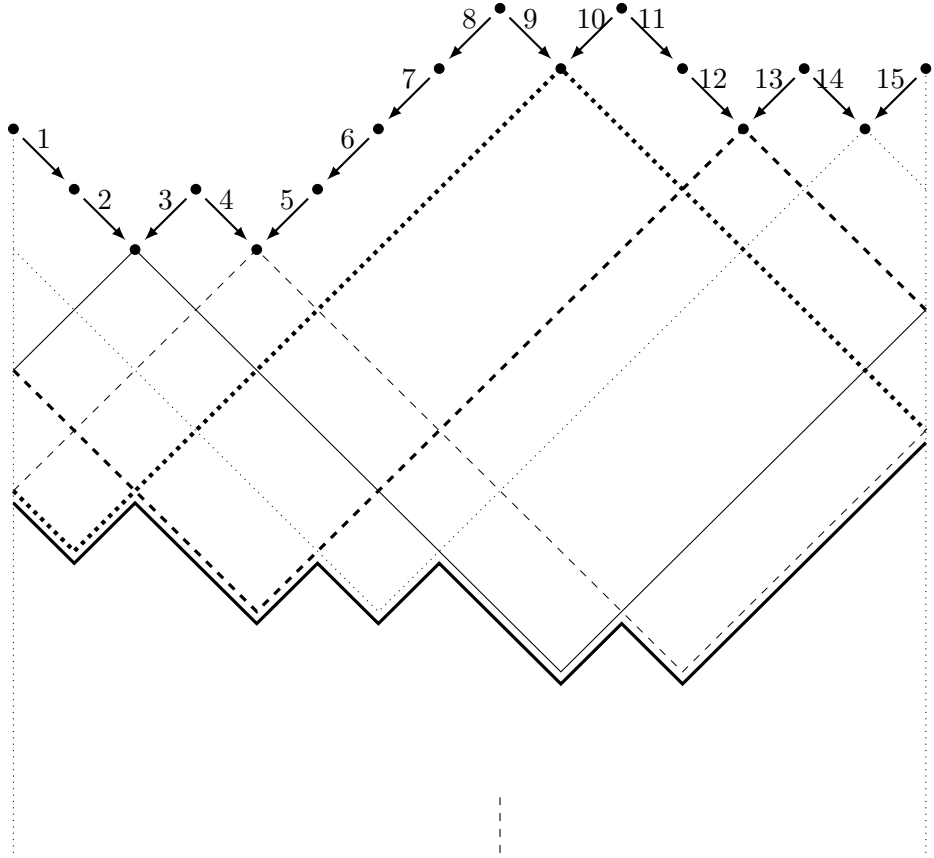
\[
\begin{tikzpicture}[scale=0.8,baseline=(bb.base),
quivarrow/.style={black, -latex, thick}]
\newcommand{\seventh}{51.4} 
\newcommand{\circradius}{1.5cm}
\newcommand{\inradius}{1.2cm}
\newcommand{\outradius}{1.8cm}
\newcommand{\dotrad}{0.1cm} 
\newcommand{\bdrydotrad}{{0.8*\dotrad}} 
\path (0,0) node (bb) {}; 


\draw (0,2) circle(\bdrydotrad) [fill=black];
\draw (1,1) circle(\bdrydotrad) [fill=black];
\draw (2,0) circle(\bdrydotrad) [fill=black];
\draw (3,1) circle(\bdrydotrad) [fill=black];
\draw (4,0) circle(\bdrydotrad) [fill=black];
\draw (5,1) circle(\bdrydotrad) [fill=black];
\draw (6,2) circle(\bdrydotrad) [fill=black];
\draw (7,3) circle(\bdrydotrad) [fill=black];
\draw (8,4) circle(\bdrydotrad) [fill=black];
\draw (9,3) circle(\bdrydotrad) [fill=black];
\draw (10,4) circle(\bdrydotrad) [fill=black];
\draw (11,3) circle(\bdrydotrad) [fill=black];
\draw (12,2) circle(\bdrydotrad) [fill=black];
\draw (13,3) circle(\bdrydotrad) [fill=black];
\draw (14,2) circle(\bdrydotrad) [fill=black];
\draw (15,3) circle(\bdrydotrad) [fill=black];




\draw [quivarrow,shorten <=5pt, shorten >=5pt] (0,2)
-- node[above]{$1$} (1,1);
\draw [quivarrow,shorten <=5pt, shorten >=5pt] (1,1) -- node[above]{$2$} (2,0);
\draw [quivarrow,shorten <=5pt, shorten >=5pt] (3,1) -- node[above]{$3$} (2,0);
\draw [quivarrow,shorten <=5pt, shorten >=5pt] (3,1) -- node[above]{$4$} (4,0);
\draw [quivarrow,shorten <=5pt, shorten >=5pt] (5,1) -- node[above]{$5$} (4,0);
\draw [quivarrow,shorten <=5pt, shorten >=5pt] (6,2) -- node[above]{$6$} (5,1);
\draw [quivarrow,shorten <=5pt, shorten >=5pt] (7,3) -- node[above]{$7$} (6,2);
\draw [quivarrow,shorten <=5pt, shorten >=5pt] (8,4) -- node[above]{$8$} (7,3);
\draw [quivarrow,shorten <=5pt, shorten >=5pt] (8,4) -- node[above]{$9$} (9,3);
\draw [quivarrow,shorten <=5pt, shorten >=5pt] (10,4) -- node[above]{$10$} (9,3);
\draw [quivarrow,shorten <=5pt, shorten >=5pt] (10,4) -- node[above]{$11$} (11,3);
\draw [quivarrow,shorten <=5pt, shorten >=5pt] (11,3) -- node[above]{$12$} (12,2);
\draw [quivarrow,shorten <=5pt, shorten >=5pt] (13,3) -- node[above]{$13\,\,$} (12,2);
\draw [quivarrow,shorten <=5pt, shorten >=5pt] (13,3) -- node[above]{$14\,\,$} (14,2);
\draw [quivarrow,shorten <=5pt, shorten >=5pt] (15,3) --node[above]{$15\,\,$}(14,2);

\draw[black][dashed, very thin] (0,-4)--(4,0)--(11,-7)--(15,-3); 
\draw[dotted, ultra thick] (0,-4)--(1,-5)--(2,-4)--(3,-3)--(4,-2)--(5,-1)--(9,3)--(11,1)--(12,0)--(13,-1)--(14,-2)--(15,-3);
\draw[dashed, very thick] (0,-2)--(4,-6)--(12,2)--(15,-1);
\draw[dotted, thin] (0,0)--(6,-6)--(14,2)--(15,1);
\draw[thin] (0,-2)--(2,0)--(9,-7)--(15,-1);


\draw[very thick] (0,-4.2)--(1,-5.2)--(2,-4.2)--(4,-6.2)--(5,-5.2)--(6,-6.2)--(7,-5.2)--(9,-7.2)--(10,-6.2)--(11,-7.2)--(15,-3.2);

\draw [dotted] (0,-10) -- (0,2);
\draw [dotted] (15,-10) -- (15,3);

\draw [dashed] (8,-10) -- (8,-9);

\end{tikzpicture}
\]
\vspace{2mm}

Now, the rim of the area below all rims of the projective indecomposable modules corresponding to the valleys of $I$ is nothing but the rim $I$ shifted by $k$ to the right (or by $n-k$ to the left). This rim is depicted by the thick black line. To see this, we notice that for each valley  $v$ of the rim $I$, in order to draw a corresponding projective indecomposable module we draw a line to the right of $v$ of size $k$, and a line to the left of the size $n-k$. If we only observe lines that we draw to the right (or  to the left) of the valleys, it is obvious that we end up with the rim that is the same as the initial rim $I$, only shifted to the right by $k$ (or to the left by $n-k$).   
\end{proof}
  
\begin{re}\label{embedding} {\rm As a submodule of $\bigoplus_{v\in V} P_v$, $\Omega^2(L_I)$ is given as a diagonally embedded copy, with $\Omega^2(L_I)$ seen as a submodule of each $P_v$ by a canonical injective map given by placing the rim of $\Omega^2(L_I)$ inside the $P_v$ as high as possible. }
\end{re}

As in the case when $r=1$,  we have that $\Omega^2(L_I)\cong L_{I^2}$, where $I^2$ is a union of the following sets (with addition modulo $n$):
$$A_1^2=\{1+k,2+k,\dots, d_1+k\},$$ 
$$A_2^2=\{d_1+l_1+1+k,d_1+l_1+2+k,\dots, d_1+l_1+d_2+k\},$$
$$\vdots$$
$$A_{r+1}^2=\{\sum_{i=1}^r d_i+\sum_{i=1}^{r} l_i+1+k,\dots, \sum_{i=1}^r d_i+\sum_{i=1}^{r} l_i+d_{r+1}+k \}.$$

We obtained $I^2$ from $I$ by adding $k$ to each element in a given segment. In  other words, just as in the case when $r=1$, we obtain $I^2$ from $I$ by shifting the rim to the right by $k$. 

If we repeat this procedure, after $2t$ steps we have that $\Omega^{2t}(L_I)\cong L_{I^{2t}}$, where $I^{2t}$ is the union of the following sets:
$$A_1^{2t}=\{1+kt,\dots, d_1+kt\},$$ 
$$A_2^{2t}=\{d_1+l_1+1+kt,d_1+l_1+2+kt,\dots, d_1+l_1+d_2+kt\},$$
$$\vdots$$
$$A_{r+1}^{2t}=\{\sum_{i=1}^r d_i+\sum_{i=1}^{r} l_i+1+kt,\dots, \sum_{i=1}^r d_i+\sum_{i=1}^{r} l_i+d_{r+1}+kt  \}.$$

\begin{te}
Let $L_I$ be a rank $1$ module whose kernel of the projective cover is of rank greater than $1$. Then $\Omega^{2n/(n,k)}(L_I)\cong L_I$ and the minimal projective cover of $L_I$ is periodic with period being an even number dividing $2n/(n,k)$. Moreover, the period $m$ is given by 
\begin{equation}\label{komplikovana} m=2\min \scalebox{1.7}{\{}  t\, \vert\, \exists c\in [1,r+1]\,  {\rm s.t.}\, d_{c+i}=d_{1+i}, \, l_{c+i}=l_{1+i}\,\, ({\rm for}\,\, i\in \overline{0,r}),$$ $${\rm and}\,\,\,   kt \equiv \sum_{i=1}^{c-1} (d_i+l_i)\!\!\!\mod n \scalebox{1.7}{\}}. 
\end{equation}
\end{te}
\begin{proof} We are looking for a minimal positive integer $m=2t$ such that $I^{2t}=I$.  Looking at $I^2$ we see that the gap between $A_i^2$ and $A_{i+1}^2$ is $l_{i}$, and in the general case, the gap between $A_i^{2t}$ and $A_{i+1}^{2t}$ is $l_{i}$.  
In order to have $I^{2t}=I$, it must hold that there is some $c\in \{1,2,\dots,n\}$ such that 
$A_{1+i}^{2t}=A_{c+i}$, for $i=\overline{0,r}$, i.e.\ we must have equality of segments, and also we must have gaps between segments in the right order, i.e.\ it must hold that $l_{c+i}=l_{1+i}$, for $i=\overline{0,r}$. In other words, if $(l_1,l_2,\dots,l_{r+1})$ is an $(r+1)$-tuple of upward segment lengths, after an appropriate cyclic permutation of this tuple it must be that $(l_1,l_2,\dots,l_{r+1})=(l_c,l_{c+1}, \dots, l_{c+r})$. The same holds for downward segments.  The above conditions are equivalent to the conditions 
$$A_{c+i}=A_{1+i}^{2t}, \,\,\,\,\,\,\, l_{c+i}=l_{1+i} \,\,\, (i=\overline{0,r}),$$ or to the conditions 
$$A_{c}=A_{1}^{2t}, \,\,\,\,\,\,\, d_{c+i}=d_{1+i}, \, l_{c+i}=l_{1+i} \,\,\, (i=\overline{0,r}).$$
The equality of sets $A_c=A_1^{2t}$ holds if $d_1=d_c$ and $1+kt \equiv \sum_{i=1}^{c-1} (d_i+l_i)+1 \mod n.$

We notice here that if we take $t=n/(n,k)$, and $c=1$, then we have $d_{c+i}=d_{1+i}$ and $l_{c+i}=l_{1+i}$ for all $i$, and 
$kt \equiv 0 \mod n$, i.e.\  $A_{1+i}^{2n}=A_{1+i}$ for all $i$, and $I=I^{2n/(n,k)}$.  Thus, the upper bound for the period $m$ is $2n/(n,k)$.  It is clear that the period must be even, because every other syzygy is a rank 1 module.

\end{proof}

As a corollary to the previous theorem we immediately get a well known result. 
\begin{co} The  algebra $B$  has infinite global dimension.
\end{co}

\begin{ex}
{\rm 

Let $n=15$, $k=7$, and $I=\{1,2,4,9,11,12,14\}.$ Since $I=\{1,2\}\cup\{4\}\cup\{9\}\cup\{11,12\}\cup\{14\}$, we have that $r+1=5$, and that the $(r+1)$-tuples of lengths of  downward and upward slopes are $(2,1,1,2,1)$ and $(1,4,1,1,1)$ respectively. Since the only $c$ for which $d_1=d_c$ is either 1 or 4, we either have $A_1^m=A_1$ or $A_1^m=A_4$. Since the cyclic tuple of upward lengths starting at $A_4$ is  $(1,1,1,4,1)$ which is not equal to $(1,4,1,1,1)$, it must be the case that $A_1^m=A_1$.  By the previous theorem  it must be that  $7t\cong 0 \mod 15$.  We are left to find the smallest $t$ such that $kt\cong 0 \mod n $, which is obviously 15 since $(15,7)=1$.  It follows that the period $m$ of $L_I$ is $30$, which is the upper bound $2n$ from the previous theorem. 
}

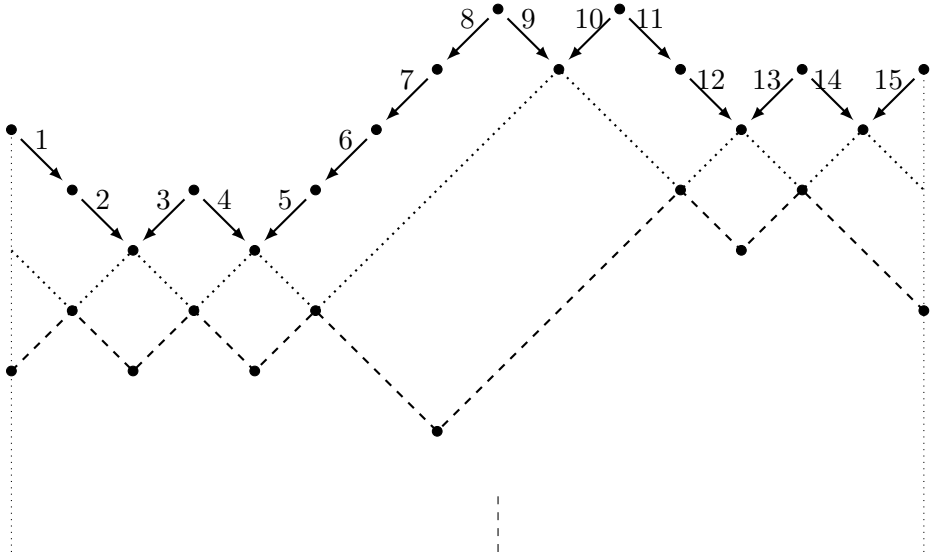
\[
\begin{tikzpicture}[scale=0.8,baseline=(bb.base),
quivarrow/.style={black, -latex, thick}]
\newcommand{\seventh}{51.4} 
\newcommand{\circradius}{1.5cm}
\newcommand{\inradius}{1.2cm}
\newcommand{\outradius}{1.8cm}
\newcommand{\dotrad}{0.1cm} 
\newcommand{\bdrydotrad}{{0.8*\dotrad}} 
\path (0,0) node (bb) {}; 


\draw (0,2) circle(\bdrydotrad) [fill=black];
\draw (1,1) circle(\bdrydotrad) [fill=black];
\draw (2,0) circle(\bdrydotrad) [fill=black];
\draw (3,1) circle(\bdrydotrad) [fill=black];
\draw (4,0) circle(\bdrydotrad) [fill=black];
\draw (5,1) circle(\bdrydotrad) [fill=black];
\draw (6,2) circle(\bdrydotrad) [fill=black];
\draw (7,3) circle(\bdrydotrad) [fill=black];
\draw (8,4) circle(\bdrydotrad) [fill=black];
\draw (9,3) circle(\bdrydotrad) [fill=black];
\draw (10,4) circle(\bdrydotrad) [fill=black];
\draw (11,3) circle(\bdrydotrad) [fill=black];
\draw (12,2) circle(\bdrydotrad) [fill=black];
\draw (13,3) circle(\bdrydotrad) [fill=black];
\draw (14,2) circle(\bdrydotrad) [fill=black];
\draw (15,3) circle(\bdrydotrad) [fill=black];


\draw (0,-2) circle(\bdrydotrad) [fill=black];
\draw (1,-1) circle(\bdrydotrad) [fill=black];
\draw (2,-2) circle(\bdrydotrad) [fill=black];
\draw (3,-1) circle(\bdrydotrad) [fill=black];
\draw (4,-2) circle(\bdrydotrad) [fill=black];
\draw (5,-1) circle(\bdrydotrad) [fill=black];
\draw (7,-3) circle(\bdrydotrad) [fill=black];
\draw (11,1) circle(\bdrydotrad) [fill=black];
\draw (12,0) circle(\bdrydotrad) [fill=black];
\draw (13,1) circle(\bdrydotrad) [fill=black];
\draw (15,-1) circle(\bdrydotrad) [fill=black];


\draw [quivarrow,shorten <=5pt, shorten >=5pt] (0,2)
-- node[above]{$1$} (1,1);
\draw [quivarrow,shorten <=5pt, shorten >=5pt] (1,1) -- node[above]{$2$} (2,0);
\draw [quivarrow,shorten <=5pt, shorten >=5pt] (3,1) -- node[above]{$3$} (2,0);
\draw [quivarrow,shorten <=5pt, shorten >=5pt] (3,1) -- node[above]{$4$} (4,0);
\draw [quivarrow,shorten <=5pt, shorten >=5pt] (5,1) -- node[above]{$5$} (4,0);
\draw [quivarrow,shorten <=5pt, shorten >=5pt] (6,2) -- node[above]{$6$} (5,1);
\draw [quivarrow,shorten <=5pt, shorten >=5pt] (7,3) -- node[above]{$7$} (6,2);
\draw [quivarrow,shorten <=5pt, shorten >=5pt] (8,4) -- node[above]{$8$} (7,3);
\draw [quivarrow,shorten <=5pt, shorten >=5pt] (8,4) -- node[above]{$9$} (9,3);
\draw [quivarrow,shorten <=5pt, shorten >=5pt] (10,4) -- node[above]{$10$} (9,3);
\draw [quivarrow,shorten <=5pt, shorten >=5pt] (10,4) -- node[above]{$11$} (11,3);
\draw [quivarrow,shorten <=5pt, shorten >=5pt] (11,3) -- node[above]{$12$} (12,2);
\draw [quivarrow,shorten <=5pt, shorten >=5pt] (13,3) -- node[above]{$13\,\,$} (12,2);
\draw [quivarrow,shorten <=5pt, shorten >=5pt] (13,3) -- node[above]{$14\,\,$} (14,2);
\draw [quivarrow,shorten <=5pt, shorten >=5pt] (15,3) --node[above]{$15\,\,$}(14,2);


\draw[dotted, thick] (0,0)--(1,-1)--(2,0)--(3,-1)--(4,0)--(5,-1)--(9,3)--(11,1)--(12,2)--(13,1)--(14,2)--(15,1);


\draw[dashed, thick] (0,-2)--(1,-1)--(2,-2)--(3,-1)--(4,-2)--(5,-1)--(7,-3)--(11,1)--(12,0)--(13,1)--(15,-1);

\draw [dotted] (0,-5) -- (0,2);
\draw [dotted] (15,-5) -- (15,3);

\draw [dashed] (8,-5) -- (8,-4);

\end{tikzpicture}
\]

\end{ex}

\begin{ex}{\rm 
Let $n=8$, $k=4$ and $I=\{1,3,5,7\}$. In this case $r=3$. From the below picture it is obvious that $\Omega^2(L_I)\cong L_I$, i.e.\ the period of $L_I$ is $m=2$, which is the lower bound from the previous theorem, i.e.\ the minimal possible value for the period of a rank 1 module.   
}

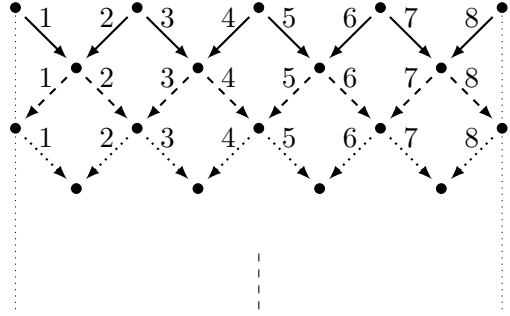
\[
\begin{tikzpicture}[scale=0.8,baseline=(bb.base),
quivarrow/.style={black, -latex, thick}]
\newcommand{\seventh}{51.4} 
\newcommand{\circradius}{1.5cm}
\newcommand{\inradius}{1.2cm}
\newcommand{\outradius}{1.8cm}
\newcommand{\dotrad}{0.1cm} 
\newcommand{\bdrydotrad}{{0.8*\dotrad}} 
\path (0,0) node (bb) {}; 


\draw (0,0) circle(\bdrydotrad) [fill=black];
\draw (0,2) circle(\bdrydotrad) [fill=black];
\draw (1,1) circle(\bdrydotrad) [fill=black];
\draw (2,0) circle(\bdrydotrad) [fill=black];
\draw (2,2) circle(\bdrydotrad) [fill=black];
\draw (3,1) circle(\bdrydotrad) [fill=black];
\draw (4,0) circle(\bdrydotrad) [fill=black];
\draw (4,2) circle(\bdrydotrad) [fill=black];
\draw (5,1) circle(\bdrydotrad) [fill=black];
\draw (6,0) circle(\bdrydotrad) [fill=black];
\draw (6,2) circle(\bdrydotrad) [fill=black];
\draw (7,1) circle(\bdrydotrad) [fill=black];
\draw (7,-1) circle(\bdrydotrad) [fill=black];
\draw (8,2) circle(\bdrydotrad) [fill=black];
\draw (8,0) circle(\bdrydotrad) [fill=black];
\draw (1,-1) circle(\bdrydotrad) [fill=black];
\draw (3,-1) circle(\bdrydotrad) [fill=black];
\draw (5,-1) circle(\bdrydotrad) [fill=black];


\draw [quivarrow,shorten <=5pt, shorten >=5pt] (0,2)-- node[above]{$1$} (1,1);
\draw [quivarrow,shorten <=5pt, shorten >=5pt,dashed] (1,1) -- node[above]{$1$} (0,0);
\draw [quivarrow,shorten <=5pt, shorten >=5pt] (2,2) -- node[above]{$2$} (1,1);
\draw [quivarrow,shorten <=5pt, shorten >=5pt,dashed] (1,1) -- node[above]{$2$} (2,0);
\draw [quivarrow,shorten <=5pt, shorten >=5pt] (2,2) -- node[above]{$3$} (3,1);
\draw [quivarrow,shorten <=5pt, shorten >=5pt,dashed] (3,1) -- node[above]{$3$} (2,0);
\draw [quivarrow,shorten <=5pt, shorten >=5pt] (4,2) -- node[above]{$4$} (3,1);
\draw [quivarrow,shorten <=5pt, shorten >=5pt,dashed] (3,1) -- node[above]{$4$} (4,0);
\draw [quivarrow,shorten <=5pt, shorten >=5pt] (4,2) -- node[above]{$5$} (5,1);
\draw [quivarrow,shorten <=5pt, shorten >=5pt,dashed] (5,1) -- node[above]{$5$} (4,0);
\draw [quivarrow,shorten <=5pt, shorten >=5pt] (6,2) -- node[above]{$6$} (5,1);
\draw [quivarrow,shorten <=5pt, shorten >=5pt,dashed] (5,1) -- node[above]{$6$} (6,0);
\draw [quivarrow,shorten <=5pt, shorten >=5pt] (6,2) -- node[above]{$7$} (7,1);
\draw [quivarrow,shorten <=5pt, shorten >=5pt, dashed] (7,1) -- node[above]{$7$} (6,0);
\draw [quivarrow,shorten <=5pt, shorten >=5pt,dashed] (7,1) -- node[above]{$8$} (8,0);
\draw [quivarrow,shorten <=5pt, shorten >=5pt] (8,2) -- node[above]{$8$} (7,1);
\draw [quivarrow,shorten <=5pt, shorten >=5pt, dotted] (6,0) -- node[above]{$7$} (7,-1);
\draw [quivarrow,shorten <=5pt, shorten >=5pt, dotted] (8,0) -- node[above]{$8$} (7,-1);

\draw [quivarrow,shorten <=5pt, shorten >=5pt,dotted] (0,0)-- node[above]{$1$} (1,-1);
\draw [quivarrow,shorten <=5pt, shorten >=5pt,dotted] (2,0) -- node[above]{$2$} (1,-1);
\draw [quivarrow,shorten <=5pt, shorten >=5pt,dotted] (4,0) -- node[above]{$4$} (3,-1);
\draw [quivarrow,shorten <=5pt, shorten >=5pt,dotted] (2,0) -- node[above]{$3$} (3,-1);
\draw [quivarrow,shorten <=5pt, shorten >=5pt,dotted] (4,0) -- node[above]{$5$} (5,-1);
\draw [quivarrow,shorten <=5pt, shorten >=5pt,dotted] (6,0) -- node[above]{$6$} (5,-1);

\draw [dotted] (0,-3) -- (0,2);
\draw [dotted] (8,-3) -- (8,2);

\draw [dashed] (4,-3) -- (4,-2);

\end{tikzpicture}
\]

\end{ex}

\begin{ex}

{\rm  Let $n=12$, $k=8$, and $I=\{1,2,4,5,7,8,10,11\}.$ Since $I=\{1,2\}\cup\{4,5\}\cup\{7,8\}\cup\{10,11\}$, we have that $r+1=4$, and that the $(r+1)$-tuples of lengths of  downward and upward slopes are $(2,2,2,2)$ and $(1,1,1,1)$. Since   $d_i=d_j$ for all $i,j$, we have that $A_1^m$ could be any of the $A_i$.

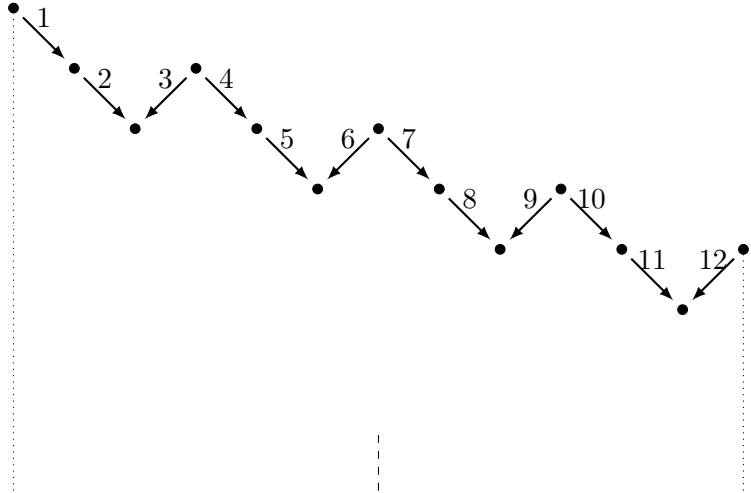
\[
\begin{tikzpicture}[scale=0.8,baseline=(bb.base),
quivarrow/.style={black, -latex, thick}]
\newcommand{\seventh}{51.4} 
\newcommand{\circradius}{1.5cm}
\newcommand{\inradius}{1.2cm}
\newcommand{\outradius}{1.8cm}
\newcommand{\dotrad}{0.1cm} 
\newcommand{\bdrydotrad}{{0.8*\dotrad}} 
\path (0,0) node (bb) {}; 


\draw (0,5) circle(\bdrydotrad) [fill=black];
\draw (1,4) circle(\bdrydotrad) [fill=black];
\draw (2,3) circle(\bdrydotrad) [fill=black];
\draw (3,4) circle(\bdrydotrad) [fill=black];
\draw (4,3) circle(\bdrydotrad) [fill=black];
\draw (5,2) circle(\bdrydotrad) [fill=black];
\draw (6,3) circle(\bdrydotrad) [fill=black];
\draw (7,2) circle(\bdrydotrad) [fill=black];
\draw (8,1) circle(\bdrydotrad) [fill=black];
\draw (9,2) circle(\bdrydotrad) [fill=black];
\draw (10,1) circle(\bdrydotrad) [fill=black];
\draw (11,0) circle(\bdrydotrad) [fill=black];
\draw (12,1) circle(\bdrydotrad) [fill=black];




\draw [quivarrow,shorten <=5pt, shorten >=5pt] (0,5)
-- node[above]{$1$} (1,4);
\draw [quivarrow,shorten <=5pt, shorten >=5pt] (1,4) -- node[above]{$2$} (2,3);
\draw [quivarrow,shorten <=5pt, shorten >=5pt] (3,4) -- node[above]{$3$} (2,3);
\draw [quivarrow,shorten <=5pt, shorten >=5pt] (3,4) -- node[above]{$4$} (4,3);
\draw [quivarrow,shorten <=5pt, shorten >=5pt] (4,3) -- node[above]{$5$} (5,2);
\draw [quivarrow,shorten <=5pt, shorten >=5pt] (6,3) -- node[above]{$6$} (5,2);
\draw [quivarrow,shorten <=5pt, shorten >=5pt] (6,3) -- node[above]{$7$} (7,2);
\draw [quivarrow,shorten <=5pt, shorten >=5pt] (7,2) -- node[above]{$8$} (8,1);
\draw [quivarrow,shorten <=5pt, shorten >=5pt] (9,2) -- node[above]{$9$} (8,1);
\draw [quivarrow,shorten <=5pt, shorten >=5pt] (9,2) -- node[above]{$10$} (10,1);
\draw [quivarrow,shorten <=5pt, shorten >=5pt] (10,1) -- node[above]{$11$} (11,0);
\draw [quivarrow,shorten <=5pt, shorten >=5pt] (12,1) -- node[above]{$12$} (11,0);





\draw [dotted] (0,-3) -- (0,5);
\draw [dotted] (12,-3) -- (12,1);

\draw [dashed] (6,-3) -- (6,-2);

\end{tikzpicture}
\]

Since the cyclic tuple of upward lengths starting at any $A_i$ is  $(1,1,1,1)$, the gaps between segments of $I$ are  in the right order for any $\Omega^m(L)$. The only condition from equation $(\ref{komplikovana})$ from the previous theorem to be fulfilled is that $A_1^m=A_c$ for some  $c$, i.e.\ that  
$$8t\equiv \sum_{i=1}^{c-1} (d_i+l_i)\!\!\!\mod 12. $$
From $d_i+l_i=3$  we have $\sum_{i=1}^{c-1} (d_i+l_i)=3(c-1).$ Thus, if $A_1^{2t}=A_c$, it must be 
$$8t \equiv 3(c-1)\!\!\! \mod 	12.$$ 
If $t=1$, then $8\not\equiv 3(c-1)\mod 12 $, for all $c$. If $t=2$, then $16\not\equiv 3(c-1)\!\!\!\mod 12 $, for all $c$. If $t=3$, then $24\equiv 3(c-1)\!\!\!\mod 12 $, for  $c=1$.  Thus, we conclude that the period of the module $L_I$ is $6$.

}
\end{ex}

\section{Extensions between rank 1 modules}

In this section we compute all (higher) extensions ${\rm Ext}^i(L_{I},L_{J})$, as a module over the centre $\mathbb F[t]$, for arbitrary rank 1 Cohen-Macaulay $B$-modules $L_I$ and $L_J$. We give a combinatorial description and an algorithm for computation of extension spaces between rank 1 Cohen-Macaulay modules by using only combinatorics of the rims $I$ and $J$.

We again use  a projective resolution of $L_{I}$
\[
\bigoplus_{v\in V} P_v \xrightarrow{D} \bigoplus_{u\in U} P_u \to  L_{I} \to 0,
\]
where $U=\{ u\not\in I : u+1 \in I\}$ and $V=\{ v \in I : v+1 \not\in I\}$.
Recall that the matrix $D=(d_{vu})$ has only non-zero entries
when $u,v$ are adjacent in $U\cup V$ (when ordered cyclically) and that 
\begin{equation}\label{eq:Dcoeffs}
 d_{vu} = \begin{cases}
 x^{v-u} & \text{when $u$ precedes $v$,}\\
 -y^{u-v} & \text{when $u$ follows $v$,}\\
  0 & \text{otherwise.}
 \end{cases}
\end{equation}

As in \cite{JKS},  applying ${\rm Hom}(-, L_{J})$ yields
\[
\begin{tikzpicture}[scale=0.8,
 arr/.style={black, -angle 60}]
\path (0,3) node (b0) {$\displaystyle\bigoplus_{u\in U}^{\phantom{U}} {\rm Hom}({P_u},{L_{J}})$};
\path (5,0.2) node(a1) {$\displaystyle\bigoplus_{v\in V} {\rm Hom}({P_v},{L_{J}})$};
\path (5,3) node (b1) {${\rm Hom}({\Omega(L_{I}}),{L_{J}})$};
\path (5,4.6) node (c1) {$0$};
\path (10,3) node(b2) {${\rm Ext}^1({L_{I}},{L_{J}})$};
\path (12.7,3) node (by) {$0$};
\path[arr] (b0) edge (b1);
\path[arr] (b1) edge (b2);
\path[arr] (b2) edge (by);
\path[arr] (c1) edge (b1);
\path[arr] (b1) edge (a1);
\path[arr] (b0) edge node[auto] {\small{$D^*$}} (a1);
\end{tikzpicture}
\]
where matrix $D^*=(d^*_{vu})$ is given by
\begin{equation}\label{eq:D*coeffs}
 d^*_{vu} = \begin{cases}
 t^a & \text{$a=\# [u,v)\setminus J$, when $u$ precedes $v$,}\\
 -t^b & \text{$b=\# J\cap [v,u)$, when $u$ follows $v$,}\\
  0 & \text{otherwise.}
 \end{cases}
\end{equation}

Note that the exponents of the monomials in the matrix $D^*$ measure the offsets between the valleys of the rim $I$ from the rim $J$, that is the offset from the canonical position on the rim $J$ of a given valley of the rim $I$  when the corresponding projective indecomposable module is mapped canonically into $L_J$. This is the same as the sum of the sizes of the upward slopes (resp.\ downward slopes) of the rim $J$ between $u$ and $v$ (resp.\ $v$ and $u$) if $u$ precedes $v$ (resp.\ if $v$ precedes $u$). If this number is 0, then the rims $I$ and $J$ have the same tendencies between $u$ and $v$ (both rims are either upward or downward sloping), and in this case the corresponding entry of the matrix $D^*$ is $1$ or $-1$.         

Also, since  ${\rm Hom}({\Omega(L_{I}}),{L_{J}})$ is a free module of rank $r$ over the centre, and ${\rm im}\, D^*$ is also a rank $r$ submodule of a free $\mathbb F[t]$-module $\displaystyle\bigoplus_{v\in V} {\rm Hom}({P_v},{L_{J}})$, we are left to compute invariant factors of $D^*$ to determine generators of a free submodule ${\rm im}\, D^*$ of $\displaystyle\bigoplus_{v\in V} {\rm Hom}({P_v},{L_{J}})$. 

Before proceeding, we note that the leading coefficient of the monomial  $d_{uv}^*$ is $1$  if $u$ is to the left of $v$ in the cyclic ordering drawn in the plane, otherwise it is $-1$.

\subsection{LR trapezia from the rims}
In this subsection we introduce a new combinatorial structure, consisting of a sequence of trapezia, that will enable us to describe extension spaces between rank 1 Cohen-Macaulay modules purely in terms of their rims. 

Let us draw the rims of $L_I$ and $L_J$ one below the other, with the rim of $L_I$ above. It does not matter how far apart vertically we draw the rims, but we demand that the rim of $L_J$ is strictly below the lowest point of the rim of $L_I$.

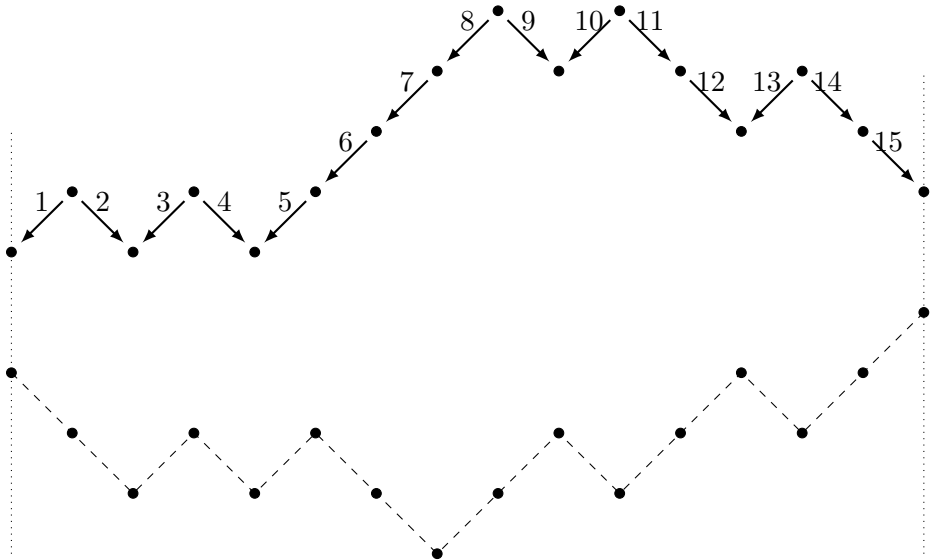
\[
\begin{tikzpicture}[scale=0.8,baseline=(bb.base),
quivarrow/.style={black, -latex, thick}]
\newcommand{\seventh}{51.4} 
\newcommand{\circradius}{1.5cm}
\newcommand{\inradius}{1.2cm}
\newcommand{\outradius}{1.8cm}
\newcommand{\dotrad}{0.1cm} 
\newcommand{\bdrydotrad}{{0.8*\dotrad}} 
\path (0,0) node (bb) {}; 


\draw (0,0) circle(\bdrydotrad) [fill=black];
\draw (1,1) circle(\bdrydotrad) [fill=black];
\draw (2,0) circle(\bdrydotrad) [fill=black];
\draw (3,1) circle(\bdrydotrad) [fill=black];
\draw (4,0) circle(\bdrydotrad) [fill=black];
\draw (5,1) circle(\bdrydotrad) [fill=black];
\draw (6,2) circle(\bdrydotrad) [fill=black];
\draw (7,3) circle(\bdrydotrad) [fill=black];
\draw (8,4) circle(\bdrydotrad) [fill=black];
\draw (9,3) circle(\bdrydotrad) [fill=black];
\draw (10,4) circle(\bdrydotrad) [fill=black];
\draw (11,3) circle(\bdrydotrad) [fill=black];
\draw (12,2) circle(\bdrydotrad) [fill=black];
\draw (13,3) circle(\bdrydotrad) [fill=black];
\draw (14,2) circle(\bdrydotrad) [fill=black];
\draw (15,1) circle(\bdrydotrad) [fill=black];


\draw (0,-2) circle(\bdrydotrad) [fill=black];
\draw (1,-3) circle(\bdrydotrad) [fill=black];
\draw (2,-4) circle(\bdrydotrad) [fill=black];
\draw (3,-3) circle(\bdrydotrad) [fill=black];
\draw (4,-4) circle(\bdrydotrad) [fill=black];
\draw (5,-3) circle(\bdrydotrad) [fill=black];
\draw (7,-5) circle(\bdrydotrad) [fill=black];
\draw (9,-3) circle(\bdrydotrad) [fill=black];
\draw (10,-4) circle(\bdrydotrad) [fill=black];
\draw (12,-2) circle(\bdrydotrad) [fill=black];
\draw (13,-3) circle(\bdrydotrad) [fill=black];
\draw (14,-2) circle(\bdrydotrad) [fill=black];
\draw (15,-1) circle(\bdrydotrad) [fill=black];
\draw (6,-4) circle(\bdrydotrad) [fill=black];
\draw (8,-4) circle(\bdrydotrad) [fill=black];
\draw (11,-3) circle(\bdrydotrad) [fill=black];


\draw [quivarrow,shorten <=5pt, shorten >=5pt] (1,1)
-- node[above]{$1$} (0,0);
\draw [quivarrow,shorten <=5pt, shorten >=5pt] (1,1) -- node[above]{$2$} (2,0);
\draw [quivarrow,shorten <=5pt, shorten >=5pt] (3,1) -- node[above]{$3$} (2,0);
\draw [quivarrow,shorten <=5pt, shorten >=5pt] (3,1) -- node[above]{$4$} (4,0);
\draw [quivarrow,shorten <=5pt, shorten >=5pt] (5,1) -- node[above]{$5$} (4,0);
\draw [quivarrow,shorten <=5pt, shorten >=5pt] (6,2) -- node[above]{$6$} (5,1);
\draw [quivarrow,shorten <=5pt, shorten >=5pt] (7,3) -- node[above]{$7$} (6,2);
\draw [quivarrow,shorten <=5pt, shorten >=5pt] (8,4) -- node[above]{$8$} (7,3);
\draw [quivarrow,shorten <=5pt, shorten >=5pt] (8,4) -- node[above]{$9$} (9,3);
\draw [quivarrow,shorten <=5pt, shorten >=5pt] (10,4) -- node[above]{$10$} (9,3);
\draw [quivarrow,shorten <=5pt, shorten >=5pt] (10,4) -- node[above]{$11$} (11,3);
\draw [quivarrow,shorten <=5pt, shorten >=5pt] (11,3) -- node[above]{$12$} (12,2);
\draw [quivarrow,shorten <=5pt, shorten >=5pt] (13,3) -- node[above]{$13\,\,$} (12,2);
\draw [quivarrow,shorten <=5pt, shorten >=5pt] (13,3) -- node[above]{$14\,\,$} (14,2);
\draw [quivarrow,shorten <=5pt, shorten >=5pt] (14,2) --node[above]{$15\,\,$}(15,1);




\draw[dashed] (0,-2)--(1,-3)--(2,-4)--(3,-3)--(4,-4)--(5,-3)--(7,-5)--(9,-3)--(10,-4)--(12,-2)--(13,-3)--(14,-2)--(15,-1);

\draw [dotted] (0,-5) -- (0,2);
\draw [dotted] (15,-5) -- (15,3);


\end{tikzpicture}
\]

We assume without loss of generality that $0\in I$, but $1\in J\setminus I$.   If we remove all the segments from both rims that are parallel, and draw vertical lines connecting the corresponding end points of the remaining segments of $I$ and $J$ we see that we are left with a collection of  trapezia.

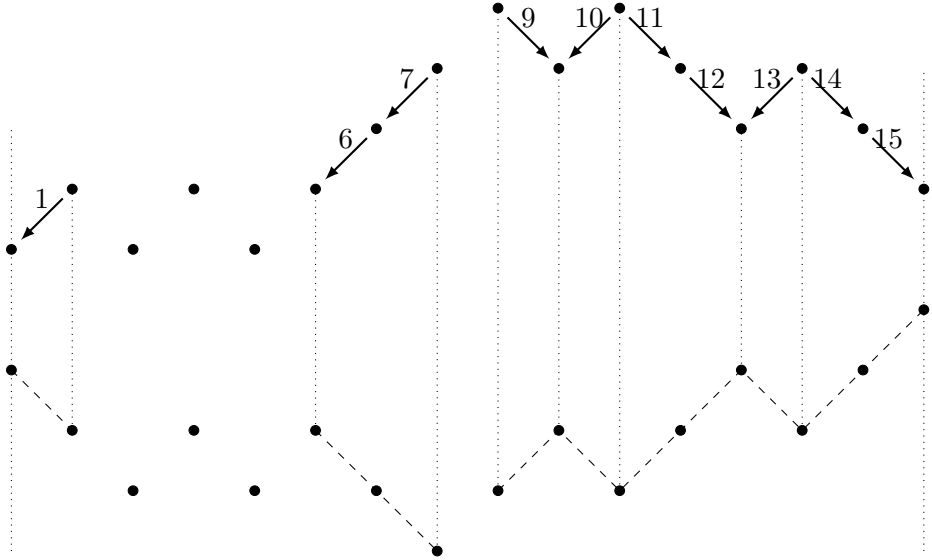
\[
\begin{tikzpicture}[scale=0.8,baseline=(bb.base),
quivarrow/.style={black, -latex, thick}]
\newcommand{\seventh}{51.4} 
\newcommand{\circradius}{1.5cm}
\newcommand{\inradius}{1.2cm}
\newcommand{\outradius}{1.8cm}
\newcommand{\dotrad}{0.1cm} 
\newcommand{\bdrydotrad}{{0.8*\dotrad}} 
\path (0,0) node (bb) {}; 


\draw (0,0) circle(\bdrydotrad) [fill=black];
\draw (1,1) circle(\bdrydotrad) [fill=black];
\draw (2,0) circle(\bdrydotrad) [fill=black];
\draw (3,1) circle(\bdrydotrad) [fill=black];
\draw (4,0) circle(\bdrydotrad) [fill=black];
\draw (5,1) circle(\bdrydotrad) [fill=black];
\draw (6,2) circle(\bdrydotrad) [fill=black];
\draw (7,3) circle(\bdrydotrad) [fill=black];
\draw (8,4) circle(\bdrydotrad) [fill=black];
\draw (9,3) circle(\bdrydotrad) [fill=black];
\draw (10,4) circle(\bdrydotrad) [fill=black];
\draw (11,3) circle(\bdrydotrad) [fill=black];
\draw (12,2) circle(\bdrydotrad) [fill=black];
\draw (13,3) circle(\bdrydotrad) [fill=black];
\draw (14,2) circle(\bdrydotrad) [fill=black];
\draw (15,1) circle(\bdrydotrad) [fill=black];


\draw (0,-2) circle(\bdrydotrad) [fill=black];
\draw (1,-3) circle(\bdrydotrad) [fill=black];
\draw (2,-4) circle(\bdrydotrad) [fill=black];
\draw (3,-3) circle(\bdrydotrad) [fill=black];
\draw (4,-4) circle(\bdrydotrad) [fill=black];
\draw (5,-3) circle(\bdrydotrad) [fill=black];
\draw (7,-5) circle(\bdrydotrad) [fill=black];
\draw (9,-3) circle(\bdrydotrad) [fill=black];
\draw (10,-4) circle(\bdrydotrad) [fill=black];
\draw (12,-2) circle(\bdrydotrad) [fill=black];
\draw (13,-3) circle(\bdrydotrad) [fill=black];
\draw (14,-2) circle(\bdrydotrad) [fill=black];
\draw (15,-1) circle(\bdrydotrad) [fill=black];
\draw (6,-4) circle(\bdrydotrad) [fill=black];

\draw (8,-4) circle(\bdrydotrad) [fill=black];
\draw (11,-3) circle(\bdrydotrad) [fill=black];


\draw [quivarrow,shorten <=5pt, shorten >=5pt] (1,1)
-- node[above]{$1$} (0,0);
\draw [quivarrow,shorten <=5pt, shorten >=5pt] (6,2) -- node[above]{$6$} (5,1);
\draw [quivarrow,shorten <=5pt, shorten >=5pt] (7,3) -- node[above]{$7$} (6,2);
\draw [quivarrow,shorten <=5pt, shorten >=5pt] (8,4) -- node[above]{$9$} (9,3);
\draw [quivarrow,shorten <=5pt, shorten >=5pt] (10,4) -- node[above]{$10$} (9,3);
\draw [quivarrow,shorten <=5pt, shorten >=5pt] (10,4) -- node[above]{$11$} (11,3);
\draw [quivarrow,shorten <=5pt, shorten >=5pt] (11,3) -- node[above]{$12$} (12,2);
\draw [quivarrow,shorten <=5pt, shorten >=5pt] (13,3) -- node[above]{$13\,\,$} (12,2);
\draw [quivarrow,shorten <=5pt, shorten >=5pt] (13,3) -- node[above]{$14\,\,$} (14,2);
\draw [quivarrow,shorten <=5pt, shorten >=5pt] (14,2) --node[above]{$15\,\,$}(15,1);




\draw[dashed] (0,-2)--(1,-3);  \draw[dashed] (5,-3)--(7,-5);  \draw[dashed] (8,-4)--(9,-3)--(10,-4)--(12,-2)--(13,-3)--(14,-2)--(15,-1);

\draw [dotted] (1,-3) -- (1,1);
\draw [dotted] (5,-3) -- (5,1);
\draw [dotted] (7,-5) -- (7,3);
\draw [dotted] (8,-4) -- (8,4);
\draw [dotted] (9,-3) -- (9,3);
\draw [dotted] (10,-4) -- (10,4);

\draw [dotted] (12,-2) -- (12,2);
\draw [dotted] (13,-3) -- (13,3);

\draw [dotted] (0,-5) -- (0,2);
\draw [dotted] (15,-5) -- (15,3);


\end{tikzpicture}
\]

 If a trapezium has a shorter base edge on its left (right) side, then we call this trapezium a left (right) trapezium.  We proceed by writing down a word containing letters $L$ and $R$ as follows: looking at the diagram of trapezia, and reading from left to right we write a letter $L$ whenever we have a left trapezium and $R$ whenever we have a right trapezium. In the above example we get the word  $w_{I,J}$: $LLRLRLR$.

Since $I\neq J$,  we can always assume, after cyclically permuting elements of $\{1,2,\dots,n\}$ if necessary, that the first letter is $L$ and that the last letter is $R$. The following step is to reduce the word $w_{I,J}$   by replacing multiple consecutive occurrences  of  $L$ (resp.\ $R$) by a single $L$ (resp.\ $R$). What we are left with is a word of the form $LRLRLRLR...LR=(LR)^s$. Let us call $s$ the rank of the reduced word $w_{I,J}$.    

If in the above diagram we treat consecutive trapezia of the same orientation as a single trapezium, then we can see the above diagram as a collection of "boxes", with box being a single pair  consisting of one left trapezium and one right trapezium. The number $s$ denotes the number of boxes for the rims $I$ and $J$.  If we ignore the parts of the rims of $I$ and $J$ that have the same tendency, then what we are left with is the two rims with always different tendencies (in other words, we have a sequence of boxes), and these rims are symmetric in the sense that one is a reflection of the other with respect to the horizontal line between them.    

\subsection{The main theorem}

Our aim is to describe extension spaces between rank 1 Cohen-Macaulay modules by using combinatorics of the corresponding rims. As a module over the centre $\mathbb F[t]$, it turns out that the extension space between rank 1 Cohen-Macaulay modules $L_I$ and $L_J$ is a torsion module isomorphic to a direct sum of the cyclic modules which are computed directly from the rims $I$ and $J$. Our main result states. 

\begin{te} \label{glavna} Let $L_I$ and $L_J$ be rank $1$ Cohen-Macaulay modules. Then, as modules over the centre $\mathbb F[t]$, 
$${\rm Ext}^1(L_I,L_J)\cong \mathbb F[t]/(t^{h_1})\times \mathbb F[t]/(t^{h_2})\times \cdots \times \mathbb F[t]/(t^{h_{s-1}}),$$ where $s$ is equal to the number of $LR$ trapezia for the rims $I$ and $J$ ($s$ is the rank of the word $w_{I,J}$), and $t^{h_1},\dots, t^{h_{s-1}} $ are the invariant factors of the matrix $D^*$ given by {\rm (\ref{eq:D*coeffs})}, with  $h_i\geq 0$ and $h_i\leq h_{i+1}$.  
\end{te}

In the coming proof of this result we give an algorithm for the computation of the numbers $h_i$ using only rims $I$ and $J$.

If we look at the above matrix $D^*$ we see that it is of the following form (rows are indexed by the valleys of $I$, with the first valley being $v_1$, which we can assume to be 0; columns are indexed by the peaks of $I$, with $u_1$ being the first peak; note that $v_1$ precedes $u_1$):

\[
\left(
\begin{array}{cccccc}
  -t^{a_1}&   &   &&&t^{b_p}\\
  t^{b_1}&  -t^{a_2} &   &&&\\
  &  t^{b_2} & -t^{a_3}  &&&\\
  &   &   & \ddots &&\\
  &   &   &&-t^{a_{p-1}} &\\
  &   &   &&t^{b_{p-1}} &-t^{a_p} \\  
\end{array}
\right)
\]

Here, the only non-zero entries are on the main diagonal and on the lower (cyclic) subdiagonal (which contains the top right entry $d^*_{1,p}$).  There are only two non-zero entries in each column and row. Both of these entries are monomials, i.e.\ $a_i, b_i\geq 0$, and their exponents are given by the sums of the sizes of the lateral sides of the corresponding trapezia that appear in a given interval $[u,v]$ or $[v,u]$. Note that $a_i=0$ (resp.\ $b_i=0$) if and only if there are no left (resp.\ right) trapezia in the above diagram between the points $v_i$ and $u_i$ (resp.\  $u_i$ and $v_{i+1}$). Thus, the $i$th column has non-zero entries equal to $-1$ and $1$ if and only if there are no left trapezia to the left of $u_i$, and no right trapezia to the right of $u_i$. In other words, this happens if and only if $I$ and $J$ have the same tendency between $v_i$ and $u_i$ , and between $u_i$ and $v_{i+1}$.    

We note here that there can only be left trapezia present between $v_i$ and $u_i$ since the rim of $I$ has only upward tendency. Analogously, there can be only right trapezia between $u_i$ and $v_{i+1}$ since the tendency of $I$ is downwards. If there are multiple left trapezia between $v_i$ and $u_i$, then we regard them as a single trapezium with an offset given by the sum of offsets of those left trapezia. The same goes for multiple right trapezia. So we regard every peak as having at most one left, and at most one right trapezium next to it. This corresponds to the reduction step of the word $w_{I,J}$ from the previous subsection. Here, reduced letters come from the same peak. 

In what follows, we will  compute the invariant factors of the matrix $D^*$, i.e.\ we will find a diagonal matrix that is equivalent over $\mathbb F[t]$ to the matrix $D^*$. Let us now assume that a part of the matrix $D^*$ is given by

\[
\left(
\begin{array}{ccccccc}
&\vdots  &   &   &&&\\
 & t^{b_{i-2}}&  -t^{a_{i-1}} &   &&&\\
 & &  t^{b_{i-1}} & -1  &&\\
 & &   &    t^{b_{i}} &-t^{a_{i+1}}&\\
 & &   &   &\vdots &\\
 & &   &   && & \\  
\end{array}
\right)
\]

In other words, for a peak $u_i$ there are no trapezia to its left. Now, we perform elementary transformations of the matrix $D^*$ to obtain an equivalent matrix. If we multiply the $i$th column by $t^{b_{i-1}}$ and add it to the column $i-1$, and then multiply the $i$th row with $t^{b_{i}}$ and add it to the $(i+1)$th row we get that the matrix $D^*$ is equivalent to the matrix

\[
\left(
\begin{array}{ccccccc}
&\vdots  &   &   &&&\\
 & t^{b_{i-2}}&  -t^{a_{i-1}} &   &&&\\
 & &  0& -1  &&\\
 & &  t^{b_i+b_{i-1}}  & 0   &-t^{a_{i+1}}&\\
 & &   &   &\vdots &\\
 & &   &   && & \\  
\end{array}
\right)
\]

The $i$th row and column have only one non-zero entry, so after appropriate swaps of rows and columns, and multiplication by -1, we have that the matrix $D^*$ is equivalent to the matrix 

\[
\left(
\begin{array}{ccccccc}
1& & &  & &\\
&\vdots  &   &   &&&\\
 & & t^{b_{i-2}}&  -t^{a_{i-1}}    &&&\\
& & &  t^{b_i+b_{i-1}}  &  -t^{a_{i+1}}&\\
 & &   &   &\vdots &\\
 & &   &   && & \\  
\end{array}
\right)
\]

We note that this operation of finding an equivalent matrix corresponds to reducing two right triangles, which come from two consecutive peaks, to a single $R$ in the word $w_{I,J}$, and that the sum of exponents $b_i+b_{i-1}$ corresponds to the sum of lateral sides (offsets) of two consecutive right trapezia. This is because $a_i=0$ means that there is no left trapezium in the word $w_{I,J}$ coming from the peak $u_i$, so we are left with potentially two consecutive right trapezia, one coming from $u_{i-1}$ and one from $u_i$.  

Analogously, if for a peak $u_i$ there are no trapezia to its right, then a part of $D^*$ is  

\[
\left(
\begin{array}{ccccccc}
&\vdots  &   &   &&&\\
 & t^{b_{i-1}}&  -t^{a_{i}} &   &&&\\
 & &  1 &  -t^{a_{i+1}} &&\\
 & &   &    t^{b_{i+1}} & &\\
 & &   &   &\vdots &\\
 & &   &   && & \\  
\end{array}
\right)
\]

After elementary transformations over $\mathbb F[t]$ we get that $D^*$ is equivalent to the matrix

\[
\left(
\begin{array}{ccccccc}
1& & &  & &\\
&\vdots  &   &   &&&\\
 & & t^{b_{i-1}}&  -t^{a_i+a_{i+1}}    &&&\\
& & &  t^{b_{i+1}}  &  -t^{a_{i+2}}&\\
 & &   &   &\vdots &\\
 & &   &   && & \\  
\end{array}
\right)
\]

We note that this operation of finding an equivalent matrix corresponds to reducing two left triangles to a single $L$ in the word $w_{I,J}$, and that the sum of exponents $a_i+a_{i+1}$ corresponds to the sum of lateral sides (offsets) of two consecutive left trapezia. This is because $b_i=0$ means that there is no right trapezium in the word $w_{I,J}$ coming from the peak $u_i$, so we are left with potentially two consecutive left trapezia, one coming from $u_{i-1}$ and one from $u_i$.

\begin{re} \label{delete}
{\rm If we combine the previous two cases, i.e. if we have that in one column of $D$ we have that the non-zero entries are -1 and 1, we get that the matrix $D^*$

\[
\left(
\begin{array}{ccccccc}
&\vdots  &   &   &&&\\
 & t^{b_{i-1}}&  -1 &   &&&\\
 & &  1 &  -t^{a_{i+1}} &&\\
 & &   &    t^{b_{i+1}} & &\\
 & &   &   &\vdots &\\
 & &   &   && & \\  
\end{array}
\right)
\]

is equivalent to the matrix 

\[
\left(
\begin{array}{ccccccc}
1& & &  & &\\
&\vdots  &   &   &&&\\
 & & t^{b_{i-1}}&  -t^{a_{i+1}}    &&&\\
& & &  t^{b_{i+1}}  &  -t^{a_{i+2}}&\\
 & &   &   &\vdots &\\
 & &   &   && & \\  
\end{array}
\right)
\]
meaning that we could just remove the column $i$, and continue to work with the smaller matrix. }
\end{re}

If we continue to apply these elementary transformations to the matrix $D^*$, we eventually end up with a matrix of the form

\[
D_1=\left(
\begin{array}{ccccccccc}
1& & &  &&& &\\
&\ddots & &&   &   &&&\\
& &  1 & &  &&&\\
 & &   &-t^{a_{1}}&&    &&&t^{b_s}\\
 &&   &t^{b_{1}}  &  -t^{a_{2}}&&&\\
 & && &t^{b_2}      &\ddots & &\\
 & &&  & &&   t^{b_{s-2}}&-t^{a_{s-1}}&  \\  
& &&   &&&    &t^{b_{s-1}}& -t^{a_s} \\  
\end{array}
\right),
\]
where $s$  is equal to the number of the $LR$ boxes for the rims $I$ and $J$, and $a_i,b_i>0$. This follows from the fact that the elementary transformations that we did on $D^*$ correspond to the reduction steps on the word $w_{I,J}$ from the previous subsection. 

Let $E$ be the $s\times s$ submatrix of $D_1$ consisting of the last $s$ rows and columns of $D_1$. Since ${\rm im}\, D^*$ is a free submodule of corank 1 of the free module $\bigoplus_{v\in V} {\rm Hom}{(P_v},{L_{J})}$, it follows that $D^*$ is also a matrix of corank 1, and that $E$ is a matrix of corank 1. There is a linear combination over $\mathbb F[t]$ of columns of $D^*$ that is equal to zero. Moreover, at least one of the coefficients in this linear combination is equal to 1 (these are precisely the coefficients of the columns corresponding to the peaks that are placed on the rim of $J$ when $L_I$ is canonically mapped into $L_J$,  see Remark \ref{placement} below).

We can assume that this column is the last column, after possibly cyclically permuting the columns. Therefore, $D^*$ is equivalent to the matrix of the form 
\[
D_1=\left(
\begin{array}{ccccccccc}
1& & &  &&& &\\
&\ddots & &&   &   &&&\\
& &  1 & &  &&&\\
 & &   &-t^{a_{1}}&&    &&&0\\
 &&   &t^{b_{1}}  &  -t^{a_{2}}&&&&\\
 & && &t^{b_2}      &\ddots & &&\\
 & &&  & &&   t^{b_{s-2}}&-t^{a_{s-1}}& \\  
& &&   &&&    &t^{b_{s-1}}& 0\\  
\end{array}
\right),
\]

Let $h$ be $\min\{a_j,b_j\vert\,  j=1,2,\dots,s-1\}$. 
Let us assume that $h=a_i$. If $i>1$, then by multiplying the $i$th column of $E$ by  $t^{b_{i-1}-a_i}$ and adding to the $(i-1)$th column, and then multiplying the $i$th row by $t^{b_i-a_i}$ and adding this row to the $(i+1)$th row we get that $E$ is equivalent to the matrix (we also do the necessary swaps of rows and columns and multiplication by -1): 
 \[
\left(
\begin{array}{cccccccccccc}
& &  t^{a_i} & &  &&&&&&\\
 & &   &-t^{a_{1}}&&    &&&&&&0\\
 &&   &t^{b_{1}}  &  -t^{a_{2}}&&&&&&&\\
 & && &t^{b_2}      &\ddots & &&&&&\\
 & && &&&-t^{a_{i-1}}      &&&&&\\
 &&&&& &t^{b_{i-1}+b_i-a_i}      &-t^{a_{i+1}}&&&& \\
&&& & &&&  t^{b_{i+1}}&&&& \\
 & &&&&& &      &\ddots & &&\\
 &&& & &&  & &&   -t^{a_{s-2}}&&  \\  
&&& & &&  & &&   t^{b_{s-2}}&-t^{a_{s-1}}&  \\  
& &&&&&   &&&    &t^{b_{s-1}}& 0\\  
\end{array}
\right).
\]

If $i=1$, then simply by multiplying the first row by $t^{b_1-a_1}$ and adding it to the second row we obtain a matrix of the above form.

If $h=b_i$ for some $i$, then by using the analogous elementary transformations (e.g.\ if $i=s-1$, then we simply multiply the last row by $t^{a_{s-1}-{b_{s-1}}}$ and add to the row above), we obtain that $E$ is equivalent to the matrix 
 \[
\left(
\begin{array}{cccccccccccc}
& &  t^{b_i} & &  &&&&&&\\
 & &   &-t^{a_{1}}&&    &&&&&&0\\
 &&   &t^{b_{1}}  &  -t^{a_{2}}&&&&&&&\\
 & && &t^{b_2}      &\ddots & &&&&&\\
 & && &&&-t^{a_{i-1}}      &&&&&\\
 &&&&& &t^{b_{i-1}   }   &-t^{a_{i+1}+a_i-b_i}&&&& \\
&&& & &&&  t^{b_{i+1}}&&&& \\
 & &&&&& &      &\ddots & &&\\
 &&& & &&  & &&   -t^{a_{s-2}}&&  \\  
&&& & &&  & &&   t^{b_{s-2}}&-t^{a_{s-1}}&\\  
& &&&&&   &&&    &t^{b_{s-1}}& 0\\  
\end{array}
\right).
\]

By repeating these steps we finally get that $D^*$ is equivalent to the matrix 
\[
\left(
\begin{array}{cccccccc}
  1&   &  &&& &&\\
  &\ddots   &   &&&&&\\
  &   &  1 &&&&&\\
  &   &   & t^{h_1}&&&&\\
  &   &  & &&\ddots&&\\
     &   &&&&&t^{h_{s-1}}&\\
    &   &&&&&&0\\
\end{array}
\right)
\]

where $0<h_1\leq h_2\leq \dots \leq h_{s-1}$. 

Now, since ${\rm Ext}^1(L_{I},L_{J})$  is isomorphic to the quotient of a free module of rank $r$ by a free submodule of the same rank generated by invariant factors of the matrix $D^*$, that is by $1,\dots,1,t^{h_1},\dots, t^{h_{s-1}}$, we have that 
$${\rm Ext}^1(L_{I},L_{J})\cong \mathbb F[t]^{r}/ \mathbb F[t]\times\dots \times \mathbb F[t]\times t^{h_1} \mathbb F[t]\times\dots \times t^{h_{s-1}}\mathbb F[t]$$ 
$$\mathbb \cong F[t]/(t^{h_1})\times\dots \times \mathbb F[t]/(t^{h_{s-1}}),$$ where $h_1\leq h_2\leq \dots \leq h_{s-1}.$ 
Recall that $r+1$ is the  number of peaks of the rim $I$.

This proves Theorem \ref{glavna}! 

\begin{co} \label{cross} If $I\neq J$, then ${\rm Ext}^1(L_{I},L_{J})=0$ if and only if the number of $LR$ boxes is equal to $1$.
\end{co}

\begin{re}{\rm  The case when the number of  $LR$  boxes is 1 is exactly the non-crossing case from \cite{JKS}, Proposition 5.6, because existence of exactly one  box means that $I$ and $J$ are non-crossing.  }
\end{re}

\subsubsection{Algorithm} Let us assume that  the exponents $a_i$ and $b_i$ of the matrix $D^*$ are given. If we denote by ${\rm IF(D^*)}$ the set of exponents of the invariant factors of the matrix $D^*$, then the following algorithm computes ${\rm IF(D^*)}$. 

$$IF(D^*):=\{\}, H_0:=\{a_1,\dots, a_{r+1}, b_1,\dots,b_{r+1}\}$$

$i:=1$, $m:=r+1$ ($m$ is the maximal index in $H_i$) 

\noindent REPEAT 

$$h_i:=\min\, H_{i-1}, IF(D^*)=IF(D^*)\cup \{h_i\}$$

CASE 1:  $h_i=a_j$ for some $j$

If $j>1$, then 
$H_{i}=H_{i-1}\setminus \{a_j, b_j, b_{j-1}\}, b_{j-1}={b_{j-1}+b_j-a_j}$

\hspace{22.5mm} $H_i=H_i\cup \{b_{j-1}\}.$

If $j=1$, then 
$H_{i}=H_{i-1}\setminus \{a_j, b_j, b_{m}\}, b_{m}={b_{m}+b_j-a_j},$

\hspace{22.5mm} $H_i=H_i\cup \{b_{m}\}.$

$i=i+1, m=m-1$

Re-enumerate indices of elements of $H_i$, that is, for $q>j$, $a_q$ becomes 

$a_{q-1}$, and $b_q$ becomes $b_{q-1}$.

CASE 2:  $h_i=b_j$ for some $j$

If $j<m$, then 
$H_{i}=H_{i-1}\setminus \{a_j, b_j, a_{j+1}\}, a_{j+1}={a_{j+1}+a_j-b_j},$

\hspace{24mm} $H_i=H_i\cup \{a_{j+1}\}.$

If $j=m$, then $H_{i}=H_{i-1}\setminus \{a_j, b_j, a_{1}\},  a_{1}={a_{1}+a_j-b_j},$

\hspace{24mm} $H_i=H_i\cup \{a_{1}\}.$

$i=i+1, m=m-1$

Re-enumerate indices of elements of $H_i$, that is, for $q>j$, $a_q$ becomes 

$a_{q-1}$, and $b_q$ becomes $b_{q-1}$.  

\noindent UNTIL $i=r+1$

\subsubsection{An example} 
Let us continue with the example from the beginning of this section. The lengths of the lateral sides of the left trapezia are: 1,0,2,1,1.   The lengths of the lateral sides of the right trapezia are: 0,0,1,2,2.  Therefore, the matrix $D^*$ is equal to: 

\[
\left(
\begin{array}{ccccc}
  -t& 0  & 0   & 0  &t^2\\
  1& -1  & 0  &0&0\\
  0&   1& -t^{2}  &0&0\\
 0 &0&t&-t&0\\
  0&0&0&t^2&-t

\end{array}
\right)
\]

By Remark \ref{delete} we can ignore the second column, i.e.\ $D^*$ is equivalent to the matrix 
\[
\left(
\begin{array}{ccccc}
  1& 0  & 0   & 0  &0\\
  0& -t  & 0  &0&t^2\\
  0&   1& -t^{2}  &0&0\\
 0 &0&t&-t&0\\
  0&0&0&t^2&-t

\end{array}
\right)
\]
Now we multiply the second column by $t^2$, add it to the third column, then multiply the third row by $t$ and add it to the the second row, swap the appropriate rows and columns to obtain the matrix 
\[
\left(
\begin{array}{ccccc}
  1& 0  & 0   & 0  &0\\
  0& 1 & 0  &0& 0\\
  0&   0& -t^{3}  &0&t^2\\
 0 &0&t&-t&0\\
  0&0&0&t^2&-t

\end{array}
\right)
\]

We are left with monomials of positive exponent. We choose a monomial with the smallest exponent, say the one in the bottom right corner. Multiply the last column by $t$, add it to the fourth column, then multiply the last row by $t$ and add it to the third row. After row and column swaps we obtain the matrix 

\[
\left(
\begin{array}{ccccc}
  1& 0  & 0   & 0  &0\\
  0& 1 & 0  &0& 0\\
  0&0&t&0&0\\
  
  0&   0&0& -t^{3}  &t^3\\
 0 &0&0&t&-t\\

\end{array}
\right)
\]

It is now obvious that the last two columns are linearly dependent, and that the final matrix we obtain is 
\[
\left(
\begin{array}{ccccc}
  1& 0  & 0   & 0  &0\\
  0& 1 & 0  &0& 0\\
  0&0&t&0&0\\
  
  0&   0&0& t  &0\\
 0 &0&0&0&0\\

\end{array}
\right)
\]

Thus, the free submodule of rank $r=4$ module  is isomorphic to 
$$\mathbb F[t]\times \mathbb F[t]\times tF\mathbb [t]\times tF\mathbb [t].$$
Hence, it follows that 
$${\rm Ext}^1(L_I,L_J)\cong \mathbb F[t]\times \mathbb F[t]\times \mathbb F[t]\times \mathbb F[t]/\mathbb F[t]\times \mathbb F[t]\times t\mathbb F[t]\times t\mathbb F[t],$$
$${\rm Ext}^1(L_I,L_J)\cong \mathbb F[t]/(t)\times \mathbb F[t]/(t).$$

\begin{re}\label{placement}{\rm In this example, we postponed the determination of which column is linearly dependent of the other columns till the very end of our transformations. This is more practical than doing it at the beginning of the computation because it can happen that it is not so obvious how to choose the appropriate column just by using the matrix $D^*$. One can say precisely which column is linearly dependent by looking at the rims $I$ and $J$. It is a column that corresponds to a peak of the rim $I$ that ends up being placed on the rim $J$, when the lattice of $L_I$ is placed inside the lattice of  $L_J$ as far up as possible, when $L_I$ is canonically mapped into $L_J$.  In this example those are the third and the fourth column. If $c_i$ denotes the $i$th column, then 
$$t^2c_1+t^2c_2+c_3+c_4+tc_5=0.$$  
The exponents in this linear combination come from the offsets of the peaks of $I$ from the rim $J$ as seen from the following picture. The offset of $u_1$ from the rim is 2, of $u_2$ is also 2, etc. Let us explain what we mean by this offset. The space  ${\rm Hom} (P_{u_1}, L_J)\cong \mathbb  F[t]$  is generated over $\mathbb F[t]$ by the canonical map $f_{u_1,J}$ that maps $P_{u_1}$ into $L_J$  by placing the peak $u_1$ onto the rim of $J$. The offset equal to 2 means that the homomorphism in question is given by  $t^2f_{u_1,J}$.  
Now, if $a_i$ is the offset of the peak $u_i$, then the map $t^{a_i}f_{u_i,J}$ is mapped under $D^*$ to a map where the  linear coefficient of $f_{v_{i+1},J}$ is the negative of the linear coefficient of  $f_{v_{i+1},J}$ as a summand of $D^*(t^{a_{i+1}}f_{u_{i+1},J})$. Added together they give 0. By taking $f=\oplus t^{a_i} f_{u_i,J}$, we have that $D^*(f)=0$. Moreover, the only maps that are mapped to 0 by $D^*$ are multiples of this $f$. Hence, the image of $D^*$ is a free module of rank $r$.

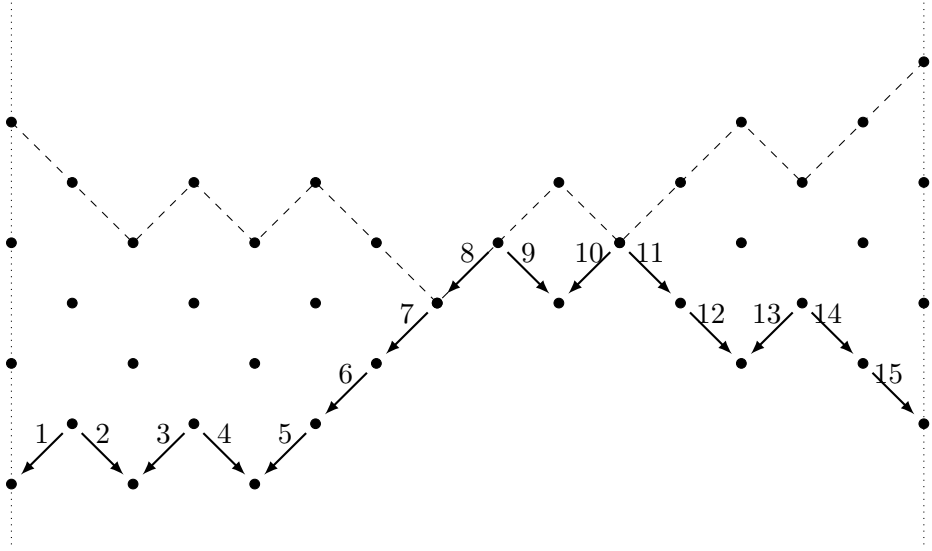
\[
\begin{tikzpicture}[scale=0.8,baseline=(bb.base),
quivarrow/.style={black, -latex, thick}]
\newcommand{\seventh}{51.4} 
\newcommand{\circradius}{1.5cm}
\newcommand{\inradius}{1.2cm}
\newcommand{\outradius}{1.8cm}
\newcommand{\dotrad}{0.1cm} 
\newcommand{\bdrydotrad}{{0.8*\dotrad}} 
\path (0,0) node (bb) {}; 


\draw (0,-7) circle(\bdrydotrad) [fill=black];
\draw (1,-6) circle(\bdrydotrad) [fill=black];
\draw (2,-7) circle(\bdrydotrad) [fill=black];
\draw (3,-6) circle(\bdrydotrad) [fill=black];
\draw (4,-7) circle(\bdrydotrad) [fill=black];
\draw (5,-6) circle(\bdrydotrad) [fill=black];
\draw (6,-5) circle(\bdrydotrad) [fill=black];
\draw (7,-4) circle(\bdrydotrad) [fill=black];
\draw (8,-3) circle(\bdrydotrad) [fill=black];
\draw (9,-4) circle(\bdrydotrad) [fill=black];
\draw (10,-3) circle(\bdrydotrad) [fill=black];
\draw (11,-4) circle(\bdrydotrad) [fill=black];
\draw (12,-5) circle(\bdrydotrad) [fill=black];
\draw (13,-4) circle(\bdrydotrad) [fill=black];
\draw (14,-5) circle(\bdrydotrad) [fill=black];
\draw (15,-6) circle(\bdrydotrad) [fill=black];


\draw (0,-1) circle(\bdrydotrad) [fill=black];
\draw (1,-2) circle(\bdrydotrad) [fill=black];
\draw (2,-3) circle(\bdrydotrad) [fill=black];
\draw (3,-2) circle(\bdrydotrad) [fill=black];
\draw (4,-3) circle(\bdrydotrad) [fill=black];
\draw (5,-2) circle(\bdrydotrad) [fill=black];
\draw (7,-4) circle(\bdrydotrad) [fill=black];
\draw (9,-2) circle(\bdrydotrad) [fill=black];
\draw (10,-3) circle(\bdrydotrad) [fill=black];
\draw (12,-1) circle(\bdrydotrad) [fill=black];
\draw (13,-2) circle(\bdrydotrad) [fill=black];
\draw (14,-1) circle(\bdrydotrad) [fill=black];
\draw (15,0) circle(\bdrydotrad) [fill=black];
\draw (6,-3) circle(\bdrydotrad) [fill=black];
\draw (11,-2) circle(\bdrydotrad) [fill=black];
\draw (12,-3) circle(\bdrydotrad) [fill=black];
\draw (14,-3) circle(\bdrydotrad) [fill=black];
\draw (15,-2) circle(\bdrydotrad) [fill=black];
\draw (15,-4) circle(\bdrydotrad) [fill=black];

\draw[black] (1,-4) circle(\bdrydotrad) [fill=black];

\draw[black] (3,-4) circle(\bdrydotrad) [fill=black];
\draw[black] (2,-5) circle(\bdrydotrad) [fill=black];
\draw[black] (0,-5) circle(\bdrydotrad) [fill=black];

\draw[black] (0,-3) circle(\bdrydotrad) [fill=black];

\draw[black] (4,-5) circle(\bdrydotrad) [fill=black];
\draw[black] (5,-4) circle(\bdrydotrad) [fill=black];


\draw [quivarrow,shorten <=5pt, shorten >=5pt] (1,-6)
-- node[above]{$1$} (0,-7);
\draw [quivarrow,shorten <=5pt, shorten >=5pt] (1,-6) -- node[above]{$2$} (2,-7);
\draw [quivarrow,shorten <=5pt, shorten >=5pt] (3,-6) -- node[above]{$3$} (2,-7);
\draw [quivarrow,shorten <=5pt, shorten >=5pt] (3,-6) -- node[above]{$4$} (4,-7);
\draw [quivarrow,shorten <=5pt, shorten >=5pt] (5,-6) -- node[above]{$5$} (4,-7);
\draw [quivarrow,shorten <=5pt, shorten >=5pt] (6,-5) -- node[above]{$6$} (5,-6);
\draw [quivarrow,shorten <=5pt, shorten >=5pt] (7,-4) -- node[above]{$7$} (6,-5);
\draw [quivarrow,shorten <=5pt, shorten >=5pt] (8,-3) -- node[above]{$8$} (7,-4);
\draw [quivarrow,shorten <=5pt, shorten >=5pt] (8,-3) -- node[above]{$9$} (9,-4);
\draw [quivarrow,shorten <=5pt, shorten >=5pt] (10,-3) -- node[above]{$10$} (9,-4);
\draw [quivarrow,shorten <=5pt, shorten >=5pt] (10,-3) -- node[above]{$11$} (11,-4);
\draw [quivarrow,shorten <=5pt, shorten >=5pt] (11,-4) -- node[above]{$12$} (12,-5);
\draw [quivarrow,shorten <=5pt, shorten >=5pt] (13,-4) -- node[above]{$13\,\,$} (12,-5);
\draw [quivarrow,shorten <=5pt, shorten >=5pt] (13,-4) -- node[above]{$14\,\,$} (14,-5);
\draw [quivarrow,shorten <=5pt, shorten >=5pt] (14,-5) --node[above]{$15\,\,$}(15,-6);




\draw[dashed] (0,-1)--(1,-2)--(2,-3)--(3,-2)--(4,-3)--(5,-2)--(7,-4)--(9,-2)--(10,-3)--(12,-1)--(13,-2)--(14,-1)--(15,0);

\draw [dotted] (0,-8) -- (0,1);
\draw [dotted] (15,-8) -- (15,1);


\end{tikzpicture}
\]

Also, we remark that even though we only used transformations on the columns that are not the last column, our transformations are valid for the last column as well. One can always think of columns being cyclically reordered so that the last column is now somewhere in the middle of the matrix.

}
\end{re}
\subsubsection{Commutativity of the ${\rm Ext}^{1}(-,-)$ functor} 

\begin{te} If $L_I$ and $L_J$ are rank $1$ modules, then, as $\mathbb F[t]$-modules
$${\rm Ext}^{1}(L_I,L_J)\cong {\rm Ext}^{1}(L_J,L_I).$$
\end{te}
\begin{proof}
Let us draw the rims of $L_I$ and $L_J$ one below the other, with the rim of $L_I$ above, and with an additional copy of the rim of $L_I$ below the rim of $L_J$. We also draw the trapezia we used to determine the extensions between the two rank 1 modules, with the upper trapezia used to compute ${\rm Ext}^{1}(L_I,L_J)$ and lower trapezia to compute ${\rm Ext}^{1}(L_J,L_I)$.

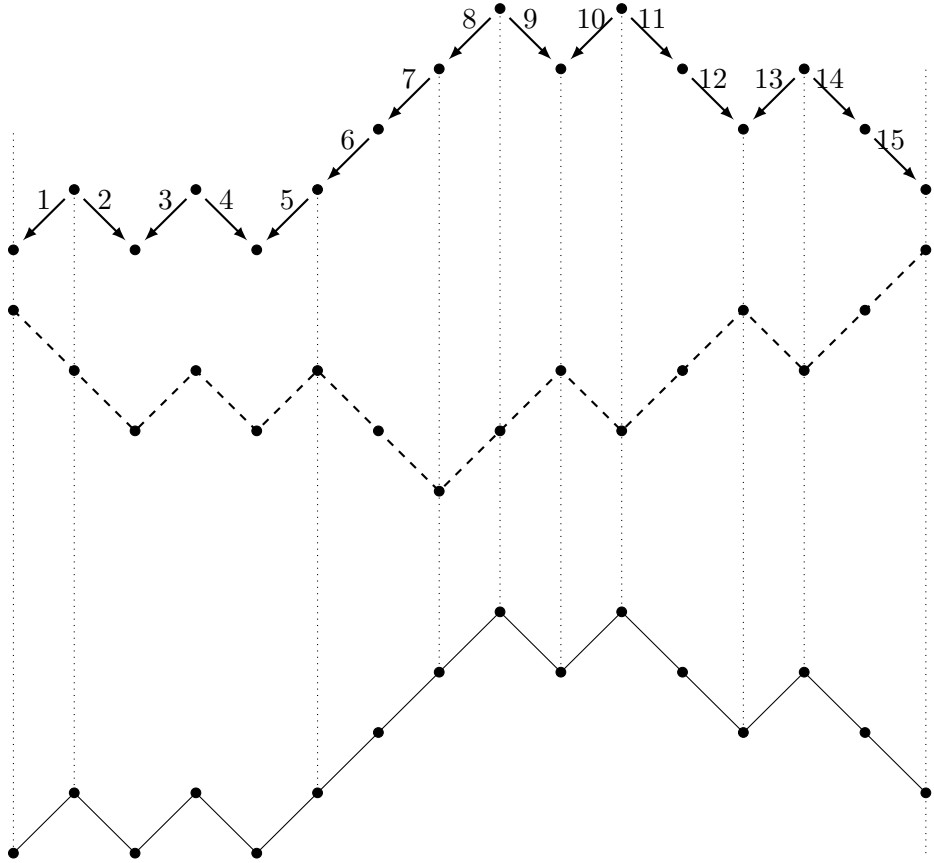
\[
\begin{tikzpicture}[scale=0.8,baseline=(bb.base),
quivarrow/.style={black, -latex, thick}]
\newcommand{\seventh}{51.4} 
\newcommand{\circradius}{1.5cm}
\newcommand{\inradius}{1.2cm}
\newcommand{\outradius}{1.8cm}
\newcommand{\dotrad}{0.1cm} 
\newcommand{\bdrydotrad}{{0.8*\dotrad}} 
\path (0,0) node (bb) {}; 


\draw (0,0) circle(\bdrydotrad) [fill=black];
\draw (1,1) circle(\bdrydotrad) [fill=black];
\draw (2,0) circle(\bdrydotrad) [fill=black];
\draw (3,1) circle(\bdrydotrad) [fill=black];
\draw (4,0) circle(\bdrydotrad) [fill=black];
\draw (5,1) circle(\bdrydotrad) [fill=black];
\draw (6,2) circle(\bdrydotrad) [fill=black];
\draw (7,3) circle(\bdrydotrad) [fill=black];
\draw (8,4) circle(\bdrydotrad) [fill=black];
\draw (9,3) circle(\bdrydotrad) [fill=black];
\draw (10,4) circle(\bdrydotrad) [fill=black];
\draw (11,3) circle(\bdrydotrad) [fill=black];
\draw (12,2) circle(\bdrydotrad) [fill=black];
\draw (13,3) circle(\bdrydotrad) [fill=black];
\draw (14,2) circle(\bdrydotrad) [fill=black];
\draw (15,1) circle(\bdrydotrad) [fill=black];


\draw (0,-1) circle(\bdrydotrad) [fill=black];
\draw (1,-2) circle(\bdrydotrad) [fill=black];
\draw (2,-3) circle(\bdrydotrad) [fill=black];
\draw (3,-2) circle(\bdrydotrad) [fill=black];
\draw (4,-3) circle(\bdrydotrad) [fill=black];
\draw (5,-2) circle(\bdrydotrad) [fill=black];
\draw (7,-4) circle(\bdrydotrad) [fill=black];
\draw (9,-2) circle(\bdrydotrad) [fill=black];
\draw (10,-3) circle(\bdrydotrad) [fill=black];
\draw (12,-1) circle(\bdrydotrad) [fill=black];
\draw (13,-2) circle(\bdrydotrad) [fill=black];
\draw (14,-1) circle(\bdrydotrad) [fill=black];
\draw (15,0) circle(\bdrydotrad) [fill=black];
\draw (6,-3) circle(\bdrydotrad) [fill=black];
\draw (8,-3) circle(\bdrydotrad) [fill=black];
\draw (11,-2) circle(\bdrydotrad) [fill=black];


\draw [quivarrow,shorten <=5pt, shorten >=5pt] (1,1)
-- node[above]{$1$} (0,0);
\draw [quivarrow,shorten <=5pt, shorten >=5pt] (1,1) -- node[above]{$2$} (2,0);
\draw [quivarrow,shorten <=5pt, shorten >=5pt] (3,1) -- node[above]{$3$} (2,0);
\draw [quivarrow,shorten <=5pt, shorten >=5pt] (3,1) -- node[above]{$4$} (4,0);
\draw [quivarrow,shorten <=5pt, shorten >=5pt] (5,1) -- node[above]{$5$} (4,0);
\draw [quivarrow,shorten <=5pt, shorten >=5pt] (6,2) -- node[above]{$6$} (5,1);
\draw [quivarrow,shorten <=5pt, shorten >=5pt] (7,3) -- node[above]{$7$} (6,2);
\draw [quivarrow,shorten <=5pt, shorten >=5pt] (8,4) -- node[above]{$8$} (7,3);
\draw [quivarrow,shorten <=5pt, shorten >=5pt] (8,4) -- node[above]{$9$} (9,3);
\draw [quivarrow,shorten <=5pt, shorten >=5pt] (10,4) -- node[above]{$10$} (9,3);
\draw [quivarrow,shorten <=5pt, shorten >=5pt] (10,4) -- node[above]{$11$} (11,3);
\draw [quivarrow,shorten <=5pt, shorten >=5pt] (11,3) -- node[above]{$12$} (12,2);
\draw [quivarrow,shorten <=5pt, shorten >=5pt] (13,3) -- node[above]{$13\,\,$} (12,2);
\draw [quivarrow,shorten <=5pt, shorten >=5pt] (13,3) -- node[above]{$14\,\,$} (14,2);
\draw [quivarrow,shorten <=5pt, shorten >=5pt] (14,2) --node[above]{$15\,\,$}(15,1);




\draw[dashed, thick] (0,-1)--(1,-2)--(2,-3)--(3,-2)--(4,-3)--(5,-2)--(7,-4)--(9,-2)--(10,-3)--(12,-1)--(13,-2)--(14,-1)--(15,0);

\draw (0,-10) circle(\bdrydotrad) [fill=black];
\draw (1,-9) circle(\bdrydotrad) [fill=black];
\draw (2,-10) circle(\bdrydotrad) [fill=black];
\draw (3,-9) circle(\bdrydotrad) [fill=black];
\draw (4,-10) circle(\bdrydotrad) [fill=black];
\draw (5,-9) circle(\bdrydotrad) [fill=black];
\draw (6,-8) circle(\bdrydotrad) [fill=black];
\draw (7,-7) circle(\bdrydotrad) [fill=black];
\draw (8,-6) circle(\bdrydotrad) [fill=black];
\draw (9,-7) circle(\bdrydotrad) [fill=black];
\draw (10,-6) circle(\bdrydotrad) [fill=black];
\draw (11,-7) circle(\bdrydotrad) [fill=black];
\draw (12,-8) circle(\bdrydotrad) [fill=black];
\draw (13,-7) circle(\bdrydotrad) [fill=black];
\draw (14,-8) circle(\bdrydotrad) [fill=black];
\draw (15,-9) circle(\bdrydotrad) [fill=black];

\draw[black] (0,-10)--(1,-9)--(2,-10)--(3,-9)--(4,-10)--(5,-9)--(6,-8)--(7,-7)--(8,-6)--(9,-7)--(10,-6)--(11,-7)--(12,-8)--(13,-7)--(14,-8)--(15,-9);

\draw [dotted] (0,-10) -- (0,2);
\draw [dotted] (15,-10) -- (15,3);



\draw [dotted] (1,-2) -- (1,1);
\draw [dotted] (5,-2) -- (5,1);
\draw [dotted] (7,-4) -- (7,3);
\draw [dotted] (8,-3) -- (8,4);
\draw [dotted] (9,-2) -- (9,3);
\draw [dotted] (10,-3) -- (10,4);

\draw [dotted] (12,-1) -- (12,2);
\draw [dotted] (13,-2) -- (13,3);


\draw [dotted] (1,-2) -- (1,-9);
\draw [dotted] (5,-2) -- (5,-9);
\draw [dotted](7,-4) -- (7,-7);
\draw [dotted] (8,-3) -- (8,-6);
\draw [dotted] (9,-2) -- (9,-7);
\draw [dotted] (10,-3) -- (10,-6);
\draw [dotted](12,-1) -- (12,-8);
\draw [dotted](13,-2) -- (13,-7);

\end{tikzpicture}
\]
\vspace{2mm}

For every left (resp.\ right) trapezium in the upper part of the above picture there is the corresponding right (resp.\ left) trapezium in the lower part of the picture. In other words, whenever $I$ and $J$ have different tendencies, it is also true for $J$ and $I$. Thus, the word consisting of $L$s and $R$s in the lower case is obtained from the word in the upper case by changing $R$s to $L$s, and $L$s to $R$s. Moreover, the corresponding trapezia are of the same lateral size, because they share a lateral side.  So, after the initial step of reducing multiple $L$s and $R$s to single $L$s and $R$s, when computing ${\rm Ext}^1(L_I,L_J)$ we get a block diagonal matrix with certain number of $1$s on the main diagonal and a matrix $A$ in the lower right corner. If we enumerate the  valleys of the rim $J$ in such a way that the first valley is the valley to the right of 0, then the corresponding matrix, obtained after the initial step of reducing multiple $L$s and $R$s to single $L$s and $R$s when computing ${\rm Ext}^1(L_J,L_I)$,  is a block diagonal matrix with certain number of $1$s on the main diagonal and with matrix $-A^t$ in the lower right corner. Since $A$ and $-A^t$ have the same set of invariant factors, it follows that ${\rm Ext}^1(L_I,L_J)\cong {\rm Ext}^1(L_J,L_I).$   

\end{proof}

\subsection{Higher extensions}
We now compute higher extensions for rank 1 Cohen-Macaulay modules. After showing how to compute higher extensions of odd degree, we prove that the even degree extensions are cyclic $\mathbb F[t]$-modules, and we show how to combinatorially compute generators of these cyclic modules. In the end we give a combinatorial criterion for vanishing of higher extension spaces between rank 1 modules. 

From the first section we know that the rank 1 modules are periodic, and moreover, for a given rim $I$,  every even syzygy in a minimal projective resolution of $L_I$ is a rank 1 module. This immediately gives us the following statement.

\begin{prop}
Let $L_I$ and $L_J$ be rank $1$ modules and $k$ a positive integer. Then there exist positive integers $h_1, h_2,\dots, h_{s-1}$, such that $h_i\leq h_{i+1}$, and, as modules over the centre $\mathbb F[t]$,  
$${\rm Ext}^{2k+1}(L_I,L_J)\cong \mathbb F[t]/(t^{h_1})\times \mathbb F[t]/(t^{h_2})\times \cdots \times \mathbb F[t]/(t^{h_{s-1}}),$$ where $s$ is equal to the number of $LR$ trapezia for the rims of $\Omega^{2k}(L_I)$ and $L_J$.  
\end{prop}
\begin{proof}
From the dimension shift formula we have that 
$${\rm Ext}^{2k+1}(L_I,L_J)\cong {\rm Ext}^1(\Omega^{2k}(L_I),J).$$ 
From the first section we know that $\Omega^{2k}(L_I)$ is a rank 1 module and the statement follows from Theorem \ref{glavna}.

\end{proof}

We are left to compute even degree extensions between rank 1 modules.  If we want to compute ${\rm Ext}^{2}(L_I,L_J)$, it is sufficient to compute ${\rm Ext}^{1}(\Omega(L_I),L_J).$
Applying ${\rm Hom}(-, L_{J})$ to the projective resolution of $\Omega(L_I)$ 

\begin{equation}\label{fff}
\xymatrix{\displaystyle\bigoplus_{w\in W} P_w \ar[r]^{F} &\displaystyle\bigoplus_{v\in V}P_v\ar[r]^D&\Omega(L_{I})}
\end{equation}

yields
\[
\begin{tikzpicture}[scale=0.8,
 arr/.style={black, -angle 60}]
\path (0,3) node (b0) {$\displaystyle\bigoplus_{v\in V}^{\phantom{U}} {\rm Hom}({P_v},{L_{J}})$};
\path (5,0.2) node(a1) {$\displaystyle\bigoplus_{w\in W} {\rm Hom}({P_w},{L_{J}})$};
\path (5,3) node (b1) {${\rm Hom}({\Omega^2(L_{I})},{L_{J}})$};
\path (5,4.6) node (c1) {$0$};
\path (10,3) node(b2) {${\rm Ext}^1({\Omega(L_{I}}),{L_{J}})$};
\path (12.7,3) node (by) {$0$};
\path[arr] (b0) edge (b1);
\path[arr] (b1) edge (b2);
\path[arr] (b2) edge (by);
\path[arr] (c1) edge (b1);
\path[arr] (b1) edge (a1);
\path[arr] (b0) edge node[auto] {\small{$F^*$}} (a1);
\end{tikzpicture}
\]
Here, $W$ is the set of the peaks of the second syzygy of $L_I$. From the previous section we know that $W=U+k$, where $U=\{u_1,u_2,\dots,u_r\}$ is the set of the cyclically ordered peaks of the rim $I$. Let $V=\{v_1,v_2,\dots,v_r\}$ be the set of cyclically ordered valleys of the rim $I$. We say that $u$ is to the left of $v$ if in the cyclic ordering of $\{1,2,\dots,n\}$, the interval $(u,v]$ does not have more than  $k$ elements. Otherwise, we say that $u$ is to the right of $v$. We assume that $u_1$ is to the right of $v_1$, and that $u_r$ is to the left of $v_1$.

\begin{te}
Let $L_I$ and $L_J$ be rank $1$ modules and $m$ a positive integer. There exists a non-negative integer $a$,  such that, as  $\mathbb F[t]$-modules,  $${\rm Ext}^{2m}(L_I,L_J)\cong \mathbb F[t]/ (t^a).$$ 
\end{te}
\begin{proof}
Using the dimension shift formula again we have that 
$${\rm Ext}^{2m}(L_I,L_J)\cong {\rm Ext}^2(\Omega^{2m-2}(L_I),J).$$
Since $\Omega^{2k-2}(L_I)$ is a rank 1 module, we are left to prove the statement for ${\rm Ext}^{2}(L_I,L_J)$.  Since 
${\rm Ext}^{2}(L_I,L_J)\cong  \overline{{\rm Hom}} (\Omega^2(L_I),L_J)$ 
it follows that ${\rm Ext}^{2m}(L_I,L_J)\cong  \mathbb F[t]/ (p(t))$ for some polynomial $p(t)$, as ${\rm Hom}(\Omega^2(L_I),L_J)$ is a rank 1 module, and $\overline{{\rm Hom}}(\Omega^2(L_I),L_J)$ is its quotient by a free submodule. We now prove that this polynomial $p(t)$ is a monomial. If we want to compute ${\rm Ext}^{2}(L_I,L_J)$, it is sufficient to compute ${\rm Ext}^{1}(\Omega(L_I),L_J)$.

From the above diagram we know that ${\rm im}\, F^*$ is a free module isomorphic to a submodule of ${\rm Hom}(\Omega^2(L_I),L_J)$. Hence, the matrix of $F^*$ is a matrix of rank 1 over $\mathbb F[t]$. Since the map $F^*$ is given by the maps from $P_w$ to $L_J$, which are given by multiplication by $t^l$ for some exponent $l$, the matrix of $F^*$ consists of the monomials. Because it is a matrix of rank 1 it follows that there is a column such that every other column is a multiple of that column. To find the invariant factor of this matrix, it remains to find a monomial with the smallest exponent from that column. This exponent gives us the integer $a$.  
\end{proof}

\begin{co}
For the integer $a$ from the previous theorem we have 
$$a=\min_{u\in U, v\in V}\{a_{uv}\} ,$$
where

$$ a_{uv}=\left\{
\begin{array}{ll}
 \# J \cap (u-(n-k), v] + \#  I\cap (v,u],& \text{if  $u$ is to the right of $v$;}   \\
\# (u,v]\setminus I + \# (v,u+k]\setminus J, &   \text{if $u$ is to the left of $v$.} \\  
\end{array} \right.
$$

\end{co}
\begin{proof} Let us first note that $W=I+k$ and label the matrix of $F^*$ with pairs $(u,v)$ rather than with pairs $(w,v)$ with $u$ corresponding to the element $u+k=w$  ($w=u-(n-k)$) of $W$. 

The numbers under the minimum function are the offsets of a given peak of $\Omega^2(L_I)$ from its canonical position when mapped into  $L_J$, that is, they give us monomials $t^{a_{uv}}$ in the matrix $F^*$ given by $(\ref{fff})$. Continuing with the example where  $I=\{1,2,4,9,11,12,14\}$ and $J=\{ 1,2,4,6,7,10,13 \}$, in the following picture (note that in the below picture the dashed and the thick black rim intersect between nodes 7 and 10, and that the dashed and thin black rim intersect between nodes 0 and 2)
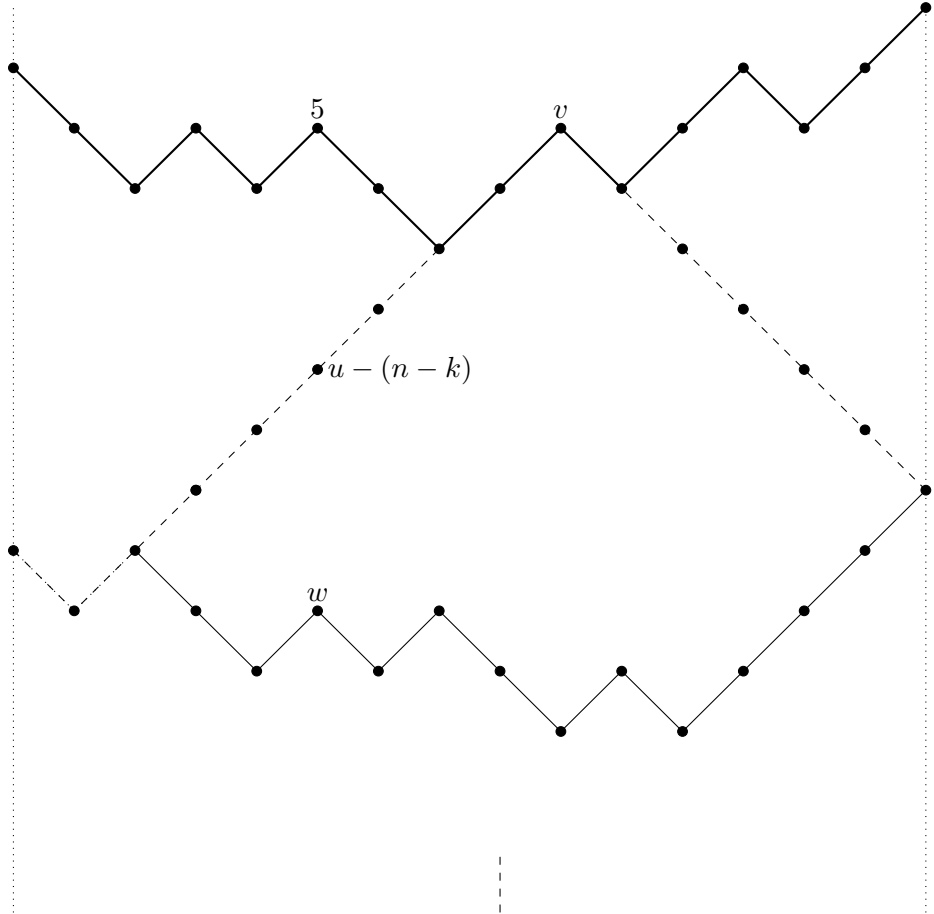
\[
\begin{tikzpicture}[scale=0.8,baseline=(bb.base),
quivarrow/.style={black, -latex, thick}]
\newcommand{\seventh}{51.4} 
\newcommand{\circradius}{1.5cm}
\newcommand{\inradius}{1.2cm}
\newcommand{\outradius}{1.8cm}
\newcommand{\dotrad}{0.1cm} 
\newcommand{\bdrydotrad}{{0.8*\dotrad}} 
\path (0,0) node (bb) {}; 

\draw (0,4) circle(\bdrydotrad) [fill=black];
\draw (1,3) circle(\bdrydotrad) [fill=black];
\draw (2,2) circle(\bdrydotrad) [fill=black];
\draw (3,3) circle(\bdrydotrad) [fill=black];
\draw (4,2) circle(\bdrydotrad) [fill=black];
\draw (5,3) circle(\bdrydotrad) [fill=black]node[above]{$5$};
\draw (5,-1) circle(\bdrydotrad) [fill=black]node[right]{$u-(n-k)$};
\draw (7,1) circle(\bdrydotrad) [fill=black];
\draw (9,3) circle(\bdrydotrad) [fill=black]node[above]{$v$};
\draw (10,2) circle(\bdrydotrad) [fill=black];
\draw (12,4) circle(\bdrydotrad) [fill=black];
\draw (13,3) circle(\bdrydotrad) [fill=black];
\draw (15,5) circle(\bdrydotrad) [fill=black];
\draw (6,2) circle(\bdrydotrad) [fill=black];
\draw (8,2) circle(\bdrydotrad) [fill=black];
\draw (11,3) circle(\bdrydotrad) [fill=black];

\draw (11,1) circle(\bdrydotrad) [fill=black];
\draw (14,4) circle(\bdrydotrad) [fill=black];
\draw (12,0) circle(\bdrydotrad) [fill=black];
\draw (13,-1) circle(\bdrydotrad) [fill=black];
\draw (14,-2) circle(\bdrydotrad) [fill=black];
\draw (15,-3) circle(\bdrydotrad) [fill=black];
\draw (6,0) circle(\bdrydotrad) [fill=black];
\draw (4,-2) circle(\bdrydotrad) [fill=black];
\draw (3,-3) circle(\bdrydotrad) [fill=black];
\draw (2,-4) circle(\bdrydotrad) [fill=black];
\draw (1,-5) circle(\bdrydotrad) [fill=black];
\draw (0,-4) circle(\bdrydotrad) [fill=black];
\draw (3,-3) circle(\bdrydotrad) [fill=black];
\draw (8,-6) circle(\bdrydotrad) [fill=black];
\draw (12,-6) circle(\bdrydotrad) [fill=black];
\draw (13,-5) circle(\bdrydotrad) [fill=black];
\draw (14,-4) circle(\bdrydotrad) [fill=black];
\draw (3,-5) circle(\bdrydotrad) [fill=black];
\draw (4,-6) circle(\bdrydotrad) [fill=black];
\draw (6,-6) circle(\bdrydotrad) [fill=black];
\draw (7,-5) circle(\bdrydotrad) [fill=black];
\draw (9,-7) circle(\bdrydotrad) [fill=black];
\draw (10,-6) circle(\bdrydotrad) [fill=black];
\draw (11,-7) circle(\bdrydotrad) [fill=black];

\draw[black, thick] (0,4)--(1,3)--(2,2)--(3,3)--(4,2)--(5,3)--(7,1)--(9,3)--(10,2)--(12,4)--(13,3)--(14,4)--(15,5);

\draw[black, thin] (10,-6)--(11,-7)--(15,-3); 
\draw[dashed] (0,-4)--(1,-5)--(2,-4)--(3,-3)--(4,-2)--(5,-1)--(9,3)--(11,1)--(12,0)--(13,-1)--(14,-2)--(15,-3);
\draw[black, thin] (2,-4)--(4,-6)--(5,-5);
\draw[black, thin] (5,-5)--(6,-6)--(7,-5); 
\draw[black, thin] (7,-5)--(9,-7)--(10,-6);
\draw[black][dotted] (0,-4)--(1,-5)--(2,-4);
\draw (5,-5) circle(\bdrydotrad) [fill=black]node[above]{$w$};


\draw [dotted] (0,-10) -- (0,5);
\draw [dotted] (15,-10) -- (15,5);

\draw [dashed] (8,-10) -- (8,-9);

\end{tikzpicture}
\]
\noindent where we have a copy of $\Omega^2(L_I)$ placed inside of $P_9$ canonically (by placing the rim of $\Omega^2(L_I)$ as far up inside the rim of $P_9$ as possible), and with $P_9$ mapped canonically into $L_J$. The number $a_{13,9}$, with 13 corresponding to $w=13-(n-k)$, measures the vertical distance between the node labelled by $w$ on the thin black rim of $\Omega^2(L_I)$ and the node labelled by 5 on the thick black rim of $L_J$. This distance is equal to the sum of the vertical distance between the node $w$ and the dashed rim of $P_9$, and the vertical distance between the dashed rim of $P_9$ and the node $5$ of the rim of $L_J$.  The vertical distance between the node $w$ and the dashed rim of $P_v$ is equal to the number of elements in the set $I\cap (v,u]$ if $u$ is to the right of $v$, and it is equal to the number of elements in the set $(u,v]\setminus I$ if $u$ is to the left of $v$.  The vertical distance between the dashed rim of $P_v$ and the node labelled by $w=u-k$ of the thick black rim of $L_J$  is equal to the number of elements in the set  $J \cap (u-(n-k), v]$  if $u$ is to the right of $v$,  and it is equal to the number of elements in the set $(v,u+k]\setminus J$   if  $u$ is to the left of $v$.

\end{proof}

\begin{re}{\rm  Let $r+1$ be the number of peaks of the rim $I$, i.e.\ assume that $F^*$ is a matrix of the format $(r+1)\times (r+1)$. From the proof of the previous corollary we see that in order to compute the smallest exponent $a$ for the entries of $F^*$, it is sufficient to compute entries of one column and one row, which means that we have to compute at most $2r+1$ entries of the matrix $F^*$ determined by $(\ref{fff})$. We pick an arbitrary row, compute its entries, and choose the minimal one. Then we compute entries of the column that contains that minimal entry. Then the exponent $a$ is the minimal entry from that column.
}
\end{re}

\begin{te}Let $I$, $U$, $V$, $W$ and $J$ be as before. Then ${\rm Ext}^{2}(L_I,L_J)=0$ if and only if there exists $v_i\in V$ such that 
$$\# J\cap (u_i-(n-k),v_i]=0 \,\,\,\,  \text{or} \,\,\,\,   \# J\cap (v_i,u_{i-1}+k]=k-(v_i-u_{i-1}) .$$
\end{te}

\begin{proof} 
From the proof of the previous theorem and corollary we know that ${\rm Ext}^{2}(L_I,L_J)=0$ if and only if there is an element of the matrix of $F^*$ that is equal to 1. This happens only if for some $u_i\in U$ and $v_j\in V$ the number $a_{u_j,v_i}$ is zero. For a given $u_j$ and $v_i$, recalling the picture from the proof of the previous corollary, $a_{u_j,v_i}=0$ if and only if both vertical distances at a given node $u_j-(n-k)$ between the rim of $P_{v_i}$ and the rim of $L_J$, and between the  rim of $P_{v_i}$ and the rim of $\Omega^2(L_I)$ inside the rim of $P_{v_i}$ are equal to 0. Obviously, this can not happen if $u_j-(n-k)$ is a node on the rim of $\Omega^2(L_I)$ that is not on the rim of $P_{v_i}$ at the same time, as in the case of the pictured node $w$ in the picture from the proof of the previous corollary. In this case, the vertical distance between $w$ on the rim of $\Omega^2(L_I)$ and $u-(n-k)=w$ on the rim of $P_{v_i}$ is strictly positive, so $a_{u_j,v_i}>0$ in this case. We conclude that if $a_{u_j,v_i}=0$, it must be that $u_j-(n-k)$ is on both the rim of $\Omega^2(L_I)$ and the rim of $P_{v_i}$. So for a given $v_i$, the only candidates $u_j$ for $a_{u_j,v_i}$ to be 0 are $u_{i}$, which is to the right of $v_i$, with the corresponding node $u_i-(n-k)$ on both rims of $\Omega^2(L_I)$ and $P_{v_i}$, and $u_{i-1}$, which is to the left of $v_i$, with the corresponding node $u_{i-1}+k$ on both rims of $\Omega^2(L_I)$ and $P_{v_i}$.  

For these two nodes $u_i-(n-k)$ and $u_{i-1}+k$, in order for the vertical distance between the rim of $P_v$ and the rim of $L_J$ to be equal to zero at the node $u_i-(n-k)$ (resp.\  $u_{i-1}+k$), it has to be that the rim $J$ has the same tendency between $u_i-(n-k)$ and $v_i$  (resp.\ between $v_i$ and $u_{i-1}+k$) as the rim of $P_{v_i}$. This means that it must be that $\# J\cap (u_i-(n-k),v_i]=0$ (resp.\ $\# J\cap (v_i,u_{i-1}+k]=k-(v_i-u_{i-1}$).

\end{proof}

\begin{re}{\rm Combined with Corollary \ref{cross} and periodicity of rank 1 modules, the previous theorem gives us a combinatorial criterion for vanishing of ${\rm Ext}^{i}(L_I,L_J)$ for arbitrary $i>0$, and for any rank 1 modules $L_I$ and $L_J$. This criterion is given purely in terms of the rims $I$ and $J$.}
\end{re}

\begin{ex}{\rm 
Take $I=\{1,2,4,9,11,12,14\}$ as in the proof of the previous corollary, and take for $J$ to be the set $\{1,2,10,12,13,14,15\}$. 

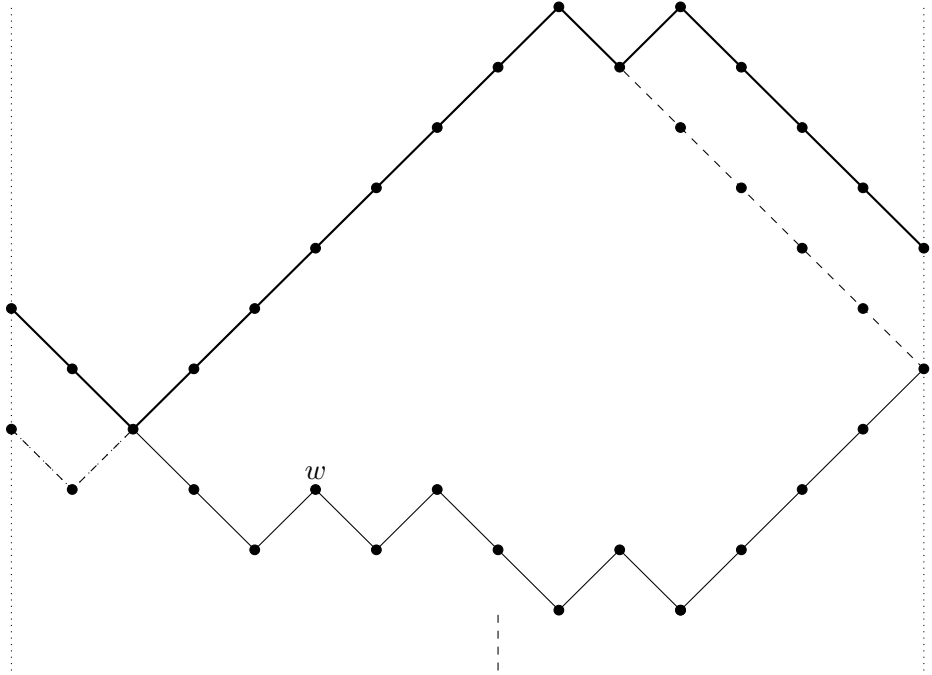
\[
\begin{tikzpicture}[scale=0.8,baseline=(bb.base),
quivarrow/.style={black, -latex, thick}]
\newcommand{\seventh}{51.4} 
\newcommand{\circradius}{1.5cm}
\newcommand{\inradius}{1.2cm}
\newcommand{\outradius}{1.8cm}
\newcommand{\dotrad}{0.1cm} 
\newcommand{\bdrydotrad}{{0.8*\dotrad}} 
\path (0,0) node (bb) {}; 


\draw (0,-2) circle(\bdrydotrad) [fill=black];
\draw (0,-4) circle(\bdrydotrad) [fill=black];
\draw (1,-5) circle(\bdrydotrad) [fill=black];
\draw (2,-4) circle(\bdrydotrad) [fill=black];
\draw (3,-3) circle(\bdrydotrad) [fill=black];
\draw (4,-2) circle(\bdrydotrad) [fill=black];
\draw (5,-1) circle(\bdrydotrad) [fill=black];
\draw (6,0) circle(\bdrydotrad) [fill=black];
\draw (7,1) circle(\bdrydotrad) [fill=black];
\draw (8,2) circle(\bdrydotrad) [fill=black];
\draw (9,3) circle(\bdrydotrad) [fill=black];
\draw (10,2) circle(\bdrydotrad) [fill=black];
\draw (11,3) circle(\bdrydotrad) [fill=black];
\draw (12,2) circle(\bdrydotrad) [fill=black];
\draw (13,1) circle(\bdrydotrad) [fill=black];
\draw (14,0) circle(\bdrydotrad) [fill=black];
\draw (15,-1) circle(\bdrydotrad) [fill=black];
\draw (15,-3) circle(\bdrydotrad) [fill=black];
\draw (14,-2) circle(\bdrydotrad) [fill=black];
\draw (13,-1) circle(\bdrydotrad) [fill=black];
\draw (12,0) circle(\bdrydotrad) [fill=black];
\draw (1,-3) circle(\bdrydotrad) [fill=black];
\draw (14,-4) circle(\bdrydotrad) [fill=black];
\draw (13,-5) circle(\bdrydotrad) [fill=black];
\draw (12,-6) circle(\bdrydotrad) [fill=black];
\draw (11,-7) circle(\bdrydotrad) [fill=black];
\draw (8,-6) circle(\bdrydotrad) [fill=black];
\draw (11,1) circle(\bdrydotrad) [fill=black];


\draw (3,-5) circle(\bdrydotrad) [fill=black];
\draw (4,-6) circle(\bdrydotrad) [fill=black];
\draw (6,-6) circle(\bdrydotrad) [fill=black];
\draw (7,-5) circle(\bdrydotrad) [fill=black];
\draw (9,-7) circle(\bdrydotrad) [fill=black];
\draw (10,-6) circle(\bdrydotrad) [fill=black];
\draw (11,-7) circle(\bdrydotrad) [fill=black];

\draw[thick] (0,-2)--(2,-4)--(9,3)--(10,2)--(11,3)--(15,-1);

\draw[black, thin] (10,-6)--(11,-7)--(15,-3); 
\draw[dashed] (0,-4)--(1,-5)--(2,-4)--(3,-3)--(4,-2)--(5,-1)--(9,3)--(11,1)--(12,0)--(13,-1)--(14,-2)--(15,-3);
\draw[black, thin] (2,-4)--(4,-6)--(5,-5);
\draw[black, thin] (5,-5)--(6,-6)--(7,-5); 
\draw[black, thin] (7,-5)--(9,-7)--(10,-6);
\draw[black][dotted] (0,-4)--(1,-5)--(2,-4);
\draw (5,-5) circle(\bdrydotrad) [fill=black]node[above]{$w$};


\draw [dotted] (0,-8) -- (0,3);
\draw [dotted] (15,-8) -- (15,3);

\draw [dashed] (8,-8) -- (8,-7);

\end{tikzpicture}
\]
From the  above picture we read-off that $a_{10,9}=0$ (note that in this picture the dashed and the thick black rim intersect between nodes 2 and 10, and that the dashed and thin black rim intersect between nodes 0 and 2). It follows that the corresponding element in the matrix of $F^*$ is equal to 1 and that ${\rm Ext}^{2}(L_I,L_J)=0.$ Note that in this case ${\rm Ext}^{1}(L_I,L_J)\neq 0$ because the number of $LR$ trapezia for the rims $I$ and $J$ is 2.

}
\end{ex}

\section*{ Acknowledgements}
We would like to thank Alastair King for all his help with this project. The authors were supported by the Austrian Science Fund project number P25647-N26, the first author was 
also supported by the project FWF W1230.

\end{document}